\documentclass[11pt,reqno]{amsart}

\usepackage{amsmath}

\usepackage{amsfonts}
\usepackage{amssymb}
\usepackage{amscd}
\usepackage{color}
\usepackage{graphicx}
\usepackage{graphbox}
\usepackage{xr-hyper}
\usepackage{xy}

\usepackage[margin=1in]{geometry}

\catcode `\@=11
\def\numberbysection{\@addtoreset{equation}{section}
         \renewcommand{\theequation}{\thesection.\arabic{equation}}}
\numberbysection
\def\subsubsection{\@startsection{subsubsection}{3}%
  \normalparindent{.5\linespacing\@plus.7\linespacing}{-.5em}%
  {\normalfont\bfseries}}
\newcounter{lstcounter}
\newtheorem{introthm}{Theorem}
\newtheorem{thm}{Theorem}[section]
\newtheorem{lem}[thm]{Lemma}
\newtheorem{prop}[thm]{Proposition}
\newtheorem{cor}[thm]{Corollary}

\theoremstyle{definition}

\newtheorem{rmk}[thm]{Remark}

\newtheorem{ex}[thm]{Example}

\def\nn{\nonumber}

\def\Z{\mathbb{Z}}

\def\ra{\rightarrow}
\def\del{\partial}

\def\D{\Delta}

\def\t{\tau}

\def\a{\alpha}

\def\g{\gamma}

\def\del{\partial}

\def\a{\alpha}

\def\g{\gamma}

\def\D{\Delta}
\def\eps{\epsilon}

\def\t{\tau}

\def\nn{\nonumber}
\def\la{\langle}
\def\ra{\rangle}

\def\sign{\mathrm{sign}}

\def\Hom{\mathrm{Hom}}

\def\sign{sign}

\def\Arc{\mathcal{A}rc}

\def\A{\mathcal{A}}

\def\Diioarc{\overline{\Arc}^{i\leftrightarrow o}}

\def\ioarc{\mathcal{A}rc^{i\leftrightarrow o}}

\def\Loop{\mathcal{L}oop}

\def\mk{\mathrm{mk}}

\def\sign{{sign}}

\def\mk{mk}

\def\Loop{\mathcal{L}oop}

\def\Arc{\mathcal{A}rc}

\def\odo{\otimes \cdots \otimes}
\def\ot{\otimes}

\def\kdk{,\dots,}

\usepackage{amscd}

\def\bt{\mathbf{t}}
\def\Cpp{C^{(2)}}
\def\Cp{C^{(1)}}
\def\sw{^{(1)}}
\def\sww{^{(2)}}

\def\CA{\Omega}
\newcommand{\leftsub}[2]{{\vphantom{#2}}_{#1}{#2}}
\newcommand{\bisub}[3]{{\vphantom{#2}}_{#1}{#2}_{#3}}

\def\N{\mathbb{N}}
\def\R{\mathbb{R}}
\def\wt{\mathrm{wt}}
\def\CF{\mathcal{F}}

\makeatletter
\let\@wraptoccontribs\wraptoccontribs
\makeatother

\begin{document}

\title[Hochschild actions]
{A detailed look on actions on Hochschild complexes especially the degree $1$ coproduct and  actions on loop spaces\\
}

\author
[R.\ M.\ Kaufmann]{Ralph M.\ Kaufmann}
\email{rkaufman@purdue.edu}
\address{Purdue University. Department of Mathematics. 150 N. University St., West Lafayette, IN 47907, USA.}
\date{ \today }

\begin{abstract}
We explain our previous results about Hochschild actions \cite{hoch1,hoch2} pertaining in particular to the coproduct which appeared in a different form  in \cite{GH} and provide a fresh look at the results.
We  recall the general action, specialize to the aforementioned coproduct and prove that the assumption of commutativity, made for convenience in \cite{hoch2}, is not needed. We give detailed background material on loop spaces, Hochschild complexes and dualizations, and  discuss details and extensions of these techniques which work for all operations of \cite{hoch1,hoch2}.

With respect to loop spaces, we show that the co--product is well defined modulo constant loops and going one step further that in the case of a graded Gorenstein Frobenius algebra, the co--product is well defined on the reduced normalized Hochschild complex.

 We   discuss several other aspects such as  ``time reversal'' duality  and several homotopies of operations induced by it. This provides a cohomology operation which is a homotopy of the anti--symmetrization of the coproduct. The obstruction again vanishes on the reduced normalized Hochschild complex if the Frobenius algebra is graded Gorenstein.
 Further structures such as  ``animation'' and  the BV structure and a coloring for operations on chains and cochains and a Gerstenhaber double bracket are briefly treated.

\end{abstract}

\maketitle


\section*{Introduction} In \cite{hoch1,hoch2}(arXiv  June 2006), we gave an action of a dg PROP of cellular chains of a CW complex, based on arc systems, on the Hochschild Cochain complex of a Frobenius algebra, algebraically realizing  and expanding  the Chas--Sullivan string topology \cite{CS} operations.
Among many other operations, this includes  a product of degree $0$, a coproduct of degree $1$ and a pre--Lie operation of degree $1$ whose cellular representatives together with the computation of composition of product and coproduct appear in \cite[Figure 4, Figure 5]{hoch1}.
The  action of the open version of the degree $1$ coproduct  was explicitly given the generalization to the open/closed context in \cite[\S 5.4.2]{ochoch}.

Such a coproduct was constructed by Goresky--Hingston \cite{GH} in a geometric setting and has found its way into  a symplectic setting \cite{Oancea}, see also \cite{NaefWill,RivWang2} for further developments after the announcement of our results presented here.
There are  precedents for such operations going back to the   basic string topology  \cite{CS} with further clarifications and developments in \cite{Sopenclosed,Ssigma,Sullivanoverview}.
For the coproduct to descend to homology of the loop space, one has to work relative to constant loops. This idea can be traced back to Sulllivan \cite{Sullconv}.
The vanishing of the obstruction for descent to cohomology has geometric meaning and has been used to
distinguish homotopy equivalent non diffeomorphic manifolds \cite{Sbasu}.

We will recall our coproduct, explain the background and give the geometric interpretations and show that the geometry of the product of \cite{GH} agrees with the our previously defined algebraic version  using the cosimplicial setup of \cite{jones,CJ}.
 Specifically, the relevant cellular chain complex whose cellular chains act as a dg PROP
on the Hochschild chain complex $CH^*(A,A)$ of a Frobenius algebra $A$ was defined in \cite[Definition 5.31]{hoch1}.
The action of cells was proven in \cite[Theorom B]{hoch2}.
Specializing to the case of $A=H^*(M)$, with $M$ a compact simply connected manifold, we make  the action
 explicit.
We show that using translation,  provided by \cite{jones,CJ}, the coproduct geometry agrees with that of \cite{GH} on the $E_1$--page. This matches with the original foliation
 geometry for the whole gamut of operations which goes back to
 \cite[\S4]{KLP}, see also \cite[\S5.11]{woods}
 and \cite[\S1 esp.\ Figure 1]{KP}.

In the course of this discussion, we give many details for the calculations and interpretations as well as generalizations that are universal and useful for other operations contained in the PROP and the algebras over it.
We prove that  the assumption of commutativity for the Frobenius algebra, made out of convenience in \cite{hoch2}, is not needed, by showing that all the equations that need to be satisfied for the action to be well defined and independent of choices hold for a general associative Frobenius algebra. This also yields a succinct formula for correlation functions in terms of 2d  OTFT correlation functions. Lastly, we consider restricting
 the PROP to operations which are already defined for associative algebras.

There are several stages to actions. The first is to simply define individual operations, the next is to give compatible operations, that is an operad, PROP or modular operad structure, and  last stage provides dg--actions. In  \cite{hoch1,hoch2} we provided modular operad actions for cell complexes from moduli spaces and dg--PROP actions for Sullivan--type surfaces yielding the dg--operations under discussion. Iterated k--fold operations, as considered also already appear in our  framework and are readily treated using our formalism.
As we prove below, if one restricts to the reduced Hochschild complex, one automatically discards constant loops and hence the results of \cite{HW} follow algebraically from our formalism.
We also discuss different methods for lifting the operations from the Frobenius algebra level to a chain level, e.g.\ from $H^*(M)$ to $C^*(M)$ and $C_*(M)$.

The individual operations inherently have a na\"ive duality in virtue of being defined as correlation functions given by switching input/output designations.
For instance, the degree $0$ product is dual to a degree $0$ co--product, which is different from the natural degree $1$ product.
However, there is PROPic ``time reversal symmetry'', basically rooted the asymmetric treatment of ``in'' and ``out'' boundaries. The prime example of being related by this symmetry are the degree $0$ product and the degree $1$ coproduct. The symmetric partners are obtained from the same underlying arc system, but differ by switiching the ``in'' and ``out'' boundaries.
 Moreover, there is  an in/out  symmetric, that is modular operad, theory used for moduli space operations \cite{hoch1,hoch2} in which this symmetry is natural. The string topology operations are recored by  applying degeneracy maps to ``outs''. This asymmetry is needed to obtain the correct dg-operations, cf.\ \cite{hoch2}.

Just as Gerstenhaber's bracket comes from the homotopy of $\cup$ and $\cup^{op}$, so too can the coproducts, as well many other operations, be  identified as operations from  homotopies  ---a point stressed by B.\ Tsygan. Our formalism also naturally identifies such homotopies for instance the co--product, its anti--symmetrization, the Gerstenhaber double bracket and with extra decorations the BV structure in the setting of ``animation'' \cite{tsyganbook}.

Finally, to incorporate the extra dualities from the na\"ive duality operadically, or better PROPicly, we introduce a two colored PROP to keep track of extra dualization which specifies $co$ (cohomological) and $ho$ (homological) inputs and outputs. In the actions this gives operations on Hochschild chains and cochains.
This subsumes the operations of \cite{RivWang} into  the correlation function formalism of \cite{hoch2}.
The details of these computations are consigned to \cite{KWang}, where, in particular, we show that the mixed $m_3$ operattions of \cite{RivWang} stem from a natural homotopy which is a double Gerstenhaber bibracket of degree $2$ in the sense of \cite{vdBbracket,turaevbracket}.

\subsection*{Organization}
The paper is organized in a formula forward way, first giving the algebraic formulas and then going deeper into their  origin which at the lowest level is rooted in the cell geometry.

 After fixing notation and giving essential remarks in
\S\ref{intropar}, we give the formula for the co--product and its boundary in \S\ref{coprodpar}, which is based on the cell   \cite[Figure 4]{hoch1} with
action according to \cite{hoch2} in Theorem \ref{coprodthm}.
With the explicit form of the action, one can see in which ways this (co)-chain operation descends to an operation in (co)homology.
This is made explicit in Proposition \ref{factorprop} and  Theorem \ref{bdactionthm}. The technical discussion on how to define the correlation functions and dualize them is contained in \S\ref{copar}.
The application to the coproduct is in \S\ref{CHcorpar} and a discussion of generalizations of the particular actions follows in \S\ref{gendisccopar}.

In \S\ref{looppar}, we review the Hochschild chain and cochain models for loop space according to \cite{jones,CJ}. This allows to complete the geometric identification of the action, and hence the coproduct, in the case of loop spaces.
We identify the constant loops with $\overline{CH}^0(C^*(M),C_*(M))$ in Proposition \ref{constloopprop} and  which allows us to deduce that the coproduct is well defined on $\bar H_*(LM)$ in  Corollary \ref{constloopcor}.
We also discuss several ways of regarding the operations we defined in other natural contexts in \S\ref{interpretpar}.
In \S\ref{geocoprodpar} and \S\ref{bdgeopar} we  give the geometry of the coproduct and its boundary terms.

 The geometry of the CW complexes and dg action of the cellular chains is discussed in \S\ref{cellpar}, wheres we also  succinctly define the action in terms of  local OTFT correctors.
 The concrete calculations are performed in \S\ref{calculationspar}. This contains the various dualizations and relaxations for the conditions of existence of the basic operations (\S\ref{calcpar}) and the proof the that commutativity assumption is superfluous; see Corollary  \ref{removeassumptioncor}, which also contains an explicit formula for the local OTFT correlators.

Several dualities are defined in  \S\ref{dualitypar}, in particular the na\"ive and ``time  reversal symmetry'', which bridges the different treatment of inputs and outputs in the actions and string topology. Further topics, such as  dualization, $A_\infty$ versions and ``animation'' are briefly discussed in \S\ref{furtherpar}. Finally, \S\ref{KWangpar} contains  a preview of the upgrading of the na\"ive duality into a colored action on Hochschild chain, cochains and the Hochschild-Tate complex and the Gerstenhaber double bracket. The full details are
relegated to \cite{KWang}.

\subsection*{Acknowledgement}
We thank Gabriel Drummond-Cole who brought the GH-Coproduct to our attention during a visit to the IBS center for Geometry and Physics. We also thank him for the ensuing discussions. We furthermore thank Alexandru Oancea for his enlightening talk on the subject and Muriel Livernet for bringing me to that talk and encouraging me further to write this exposition. It is a pleasure to thank D.\ Sullivan for his comments and remarks.
We also thank the IHES, where this note was written and the Simons Foundation for its support.
We thank M.\ Rivera and W.\ Zhang for discussions on the action on Hochschild chains and Boris Tsygan for his interest and questions about homotopies and ``animation".
Finally, we would like to thank the referee for the careful reading of the manuscript, the useful comments to improve the exposition and questions about further ramifications.

\section{ Preliminary Remarks and Notations}
\label{intropar}
\subsection{Removing assumptions}
 In  \cite{hoch1,hoch2}, we  used the notation $k$ for the coefficients thinking about fields. This made life easier, due to
the K\"unneth formula.  However, we can take $\Z$ coefficients throughout. In order to not confuse with the references, we  set $k=\Z$. This also conforms to the notation of \cite{Loday}.

In \cite{hoch2} commutativity of $A$ was assumed, see Assumption 4.1.2 of {\em loc.\ cit.}. This is not necessary as  had  been announced and detailed in several talks and discussions over the  years. Here we write out the proof. Indeed all the needed equations, cf.\ \cite[Remark 4.2]{hoch2} hold for any Frobenius algebra.   This follows from a direct verification by calculation, which is done in \S\ref{OTFTpar} and the resulting expression is \eqref{asscoreq}.
With hindsight, it also  follows  from the well--definedness of 2d Open Topological Field Theory (OTFT) and the equivalence of OTFTs with Frobenius algebras,  see Remark \ref{OTFTrmk}.

\subsection{Notation for the various complexes}
For an  $A$-$A$ bimodule $M$, we let $CH_*(A,M)$, $HH_*(A,M)$ be the Hochschild chain complex and homology and set $CH_*(A):=CH_*(A,A)$, $HH_*(A)=HH_*(A,A)$. Thus, $CH_n(A)= A^{\ot n+1}$.  For $A=1\oplus \bar A$, where $k$ is generated by the unit, the normlized complex is $\overline{CH}_n(A,M)=M\ot \bar A^{\ot n}$.

Dually, $CH^n(A,M)=Hom(A^{\otimes n},M)=M\otimes \check{A}^{\otimes n}$ denotes Hochschild co-chains and $HH^*$ Hochschild cohomology.
 An element $f\in CH^n(A, A)$ is a function $f:A^{\otimes n}\to A$. We use the short hand $CH^*(A)=CH^*(A,\check A)$ and $ HH^*(A)=HH^*(A,\check A)$.
The normalized cochains $\overline{CH}^*(A,M)$ are those functions $f(a_1 \odo a_n)$ which vanish if one of the $a_i=1$.

If $M=A$ then $CH^n(A,A)=A\otimes \check{A}^{\otimes n}$, and if  $M=\check{A}$ then
$:CH^n(A):=CH^n(A,\check{A})\simeq \check{A}^{\otimes n+1}\simeq Hom(A^{\ot n+1},k)$.
 In particular, as complexes $CH^*(A)=Hom(CH_*(A),k)$, see \cite[1.5.5]{Loday}.
 See \cite[Lemma 3.5]{hoch2}, and \cite[2.5.9]{Loday} for the relation to the cyclic complex and cyclic cohomology.

 The reduced Hochschild complex is defined as $\widetilde {CH}_n(A)=\overline {CH}_n(A)= A\ot \bar A^{\ot n}, n>0$ and $\widetilde {CH}_0=\bar A$.  Its homology is denoted by $\widetilde{HH}_n(A)$.
 The reduced complex $\widetilde{CH}^*(A)$ is the dual to $\widetilde{CH}_*(A)$ and is also the normalized complex modulo the constants in $\overline{CH}^0(A)$.

If $A$ is Frobenius, then  $A^{\ot n+1}\simeq CH_n(A)\simeq CH_n(A,\check A)\simeq \check A^{\ot n}
\simeq CH^n(A,\check A)\simeq CH^n(A)$, see \S\ref{copar} for more details. The duality $A\simeq \check A$
 extends to the duality between $CH_*(A)$ and $CH^*(A)$ as complexes.

 We will call graded algebra $A$ of finite type if all the graded pieces are finite dimensional. In this case, we consider $\check A$ as the graded dual.

\subsection{ Levels of action}
\label{threepar}

There are three levels to the actions of \cite{hoch1,hoch2}:
\subsubsection{Dg--PROP action on on $CH^*(A,A)$ for a Frobenius algebra $A$.}
This was established in \cite[Theorem B]{hoch2}. The definition of the action
 uses that linearly $CH^*(A)$ is  isomorphic to the  reduced tensor algebra $\bar TA$ on $A$.

 The operations are  defined   via correlations functions, which are morphisms $Y:(\bar TA)^{\ot n}\to k$.
 These dualize to the dg--PROP action as  detailed in \S\ref{copar}. This entails specifying inputs and outputs which yields a morphism in $Hom(CH^{\ot n_1},CH^{\ot n_2})$, for a specification of $n_1$ inputs and $n_2$ outputs with $n_1+n_2=n$.
 These are  compatible with the differentials and give a dg--PROP action if the input/output designation is the one specified by the cell model.   The asymmetric treatment of inputs and outputs gives rise to two types of duality, a na\"ive one which works on the level of operations ---allowing to assign inputs and outputs in the operations arbitrarily---- and a time reversal duality, see \S\ref{dualitypar}.

 Since the two complexes $CH^*(A)$ and $CH_*(A)$ are duals, we can furthermore identify the complexes:
 $Hom(CH^{\ot n_1},CH^{\ot n_2})\simeq  (CH^*)^{\ot n_1}\ot (CH_*)^{\ot n_2}$.
 This allows one to  dualize $CH^*$  outputs, as specified by the cell, as  $CH_*$ inputs and likewise dualize factors of $CH^*$, which are inputs according to the cell marking, to $CH_*$ outputs, augmenting the na\"ive duality.
  Structurally this is  handled by a two colored PROP, which we  introduce in \S\ref{KWangpar} ---more details and examples will be given in  \cite{KWang}.

 \subsubsection{PROP actions on $CH^*(D,\check D)$, for $D$  a quasi Frobenius algebra and a lift of the Casimir aka.\ diagonal.}
 A quasi Frobenius algebra, see  \cite[Definition 2.7]{hoch2} is a unital associative  dg algebra $(D,d)$ with a trace $\int$, i.e.\ a cyclically invariant counit, such that $\int da=0$ and $A=(H(D,d),\int_H)$ is Frobenius. In \cite[Theorem A and B]{hoch2}, we lifted the cochain operations to the cocycles of such a dg--algebra using a lift of the Casimir from $H$ to $A$.
 The prototypical example is $D=C^*(M)$, for a compact simply connected manifold $M$, with the lift being a choice of s lift of the diagonal. The cocycle condition was introduced to avoid the ambiguity introduced by the  choice of lift. This also means that the induced operations on cohomology are well defined and independent of the lift. However, fixing a lift  it  is clear that the operadic correlation functions \cite[\S2, \S2.3]{hoch2}, actually lift to all of $D$, that is to $C^*(D)\simeq CH^*(D,\check {D})$.

One  has to be careful with the  dualizations if $D$ is not finite dimensional or of finite type. In the case of $D=C^*(M)$ the (degree shifting) quasi isomorphism $CH^*(C^*(M),C^*(M))\simeq CH^{*+d}(C^*(M),C_*(M))$ was established in \cite{CJ}. The double complex gives rise to a spectral sequence, whose $E^1$--term is isomorphic to $CH^*(H^*(M),H^*(M))$ and the action via correlation functions \ gives a  PROP action on $CH^*(C^*(M),C^*(M)$ which induces the dg action on $E^1$.
The discussion pertaining to the coproduct are in \S\ref{gendisccopar} and \S\ref{interpretpar}.

\subsubsection{Subsets of operations which do not need dualization}
Finally, upon inspection of operations or sub--PROPs operads,  dualizations may not be necessary.
As remarked, e.g.\ in \cite[\S4]{hoch2}
 this is the case for the suboperad action yielding Deligne's conjecture \cite{del}.
  More details are given in \S\ref{hompar} and
\S \ref{gendisccopar}, specific, relevant examples are in \S\ref{calcpar}, while general background is discussed briefly in \S\ref{standarddecomppar} and \S\ref{calcpar}.

\subsection{Frobenius algebras}
\label{frobpar}
A Frobenius algebra $A$ is an  associative, unital (possibly $\Z/2\Z$ graded) algebra over a commutative ring $k$, with a non--degenerate even symmetric perfect pairing $\eta$, commonly written as $\la\,,\ra$ which is invariant, that is
$\la a,bc\ra=\la ab,c \ra$.

$A\ot A$  is again a Frobenius algebra with the usual multiplication $(a\ot b)(c\ot d)=ac\ot bd$,
 and  $\la a\ot b,c\ot d\ra=\la a,c\ra\la b,d\ra$ as the perfect even symmetric invariant paring.
 Here and often in the following, for simplicity, we omitted appropriate Koszul sign stemming from the use of the commutator $\t_{23}$, or simply add a $\pm$ sign.
 There are several schemes for sign conventions for operations  discussed at length in \cite{del,hoch2}, see \S\ref{cellorderpar} and \S\ref{standardpar} for the sign convention for operations.

Using these pairings, $\mu$ has an adjoint $\Delta_A$ defined by
$\la \Delta_A( a),b\otimes c\ra= \la a, bc\ra$.
The pairing $\eta$ defines a counit for this comultiplication $\eps$ via $\eps(a)=\la 1,a \ra=\la a,1\ra$. Alternate notations in use  are $\la a \ra:=\eps(a)=:\int a$. In this notation: $\la a,b\ra=\la ab \ra$.
If $A$ is not commutative, $\D_A$ is not cocommutative in general.
 The relationship between $\mu_A$ and $\Delta_A$ in this convention is:
 \begin{equation}
 \label{frobeq}
 \Delta_A(ab)=\D_A(a)(b\ot 1)=(1\ot a)\D_A(b)
 \end{equation}
 $\la\D_A(ab),c\ot  d\ra=\la ab,cd\ra=\la a,bcd\ra =\la \D_A(a), bc\ot d\ra=\la \D_A(a), (b\ot 1)(c\ot d)\ra=\la \D_A(a)(b\ot 1),c\ot d\ra$.
 As
 $\Delta_A(a)=\Delta_A(a 1)=\Delta_A(1 a)$:
\begin{equation}
\label{deltaeq}
\Delta_A(a)=\D_A(1)(a\ot 1)=(1\ot a)\D_A(1)
\end{equation}
The element $e=\mu\Delta(1)$, called the Euler element, will play an important role. The quantum dimension of $A$ is $tr(id_A)=\eps(e)$.

We set $\Delta_A(1)=:C=\sum C^{(1)}\ot C^{(2)}\in A\ot A$ and call it the Casimir element. We will use Sweedler notation throughout. In particular:
\begin{equation}
\label{coeffeq}
a=\sum \la a,\Cp\ra\Cpp=\sum \la \Cp, a\ra \Cpp
\end{equation}
Explicitly, if $A$ is free and  $e_i$  is a basis for $A$, $g_{ij}=\la e_i,e_j \ra$ and $g^{ij}$ is the inverse matrix then
$C=\sum_{i,j}  e_i \otimes e_j=\sum_i e_i\otimes e^i$ with $e^i=\sum_j e_j$ and
$e=\mu_A\circ \Delta_A(1)=\mu_A(C)=\sum_{ij}e_ie_j$.

$A$ isomorphic to its dual $\check A=Hom(A,k)$ via $a\mapsto \la a, \cdot \ra$.
Via this duality $\eta\in \check A^{\otimes 2}$ is dual to $C$. The Casimir element
$C$ allows to express the dual perfect paring on $\check \eta$ on $\check A$ via $\check\eta(\phi,\psi)=(\phi\otimes \psi)(C)$.
As usual $\Delta_A$  defines a multiplication $\mu_{\check A}$ on $\check A$ via
$(\phi\psi)(a)=(\phi\otimes \psi)(\Delta_A(a)$) and $\mu_A$ a comultiplication $\Delta_{\check A}(\phi)(a\ot b)=\phi(ab)$.

\subsubsection{Geometric/Gorenstein $A$}
\label{geosec}
In case case that $A=H^*(X)$ for a compact oriented connected $d$ dimensional Poincar\'e duality space $X$, or more generally
if $A$ is  graded Gorenstein with
socle $d$, we  set $e_0=1$ and $e_{\it top}$ the unique degree $d$ element with $\eps(e_{\it top})=1$.
In this case $e=sdim(A)e_{\it top}$, where $sdim$ is the super or $\Z/2\Z$ dimension.

In particular for $A=H^*(M)$ for a compact oriented $M$ and $\eps= \int_{[M]} . = \eps_{aug}\circ . \, \cap [M]$, where $\eps_{aug}$ is the augmentation map, then $e$ is the Euler class, $1$ is Poincar\'e dual to the fundamental class of $[M]$ and $e_{top}$ is Poincar\'e dual to a point, $\eps(e)=\chi(M)$ is the Euler-characteristic and $e=\eps(e)e_{top}$.

\section{The algebraic formula for the coproduct and it boundary}
\label{coprodpar}
The coproduct is defined by the action of a particular cell, which was already given in \cite[Figure 4]{hoch1}. It is depicted in Figure \ref{cellfig}.
We will state the algebraic result and then give the derivation of the  explicit formula for the operation
from the general setting of \cite{hoch2}. We will use the short hand $CH=CH^*(A,A)$.

 \begin{figure}
\includegraphics[align=c, width=.6\textwidth]{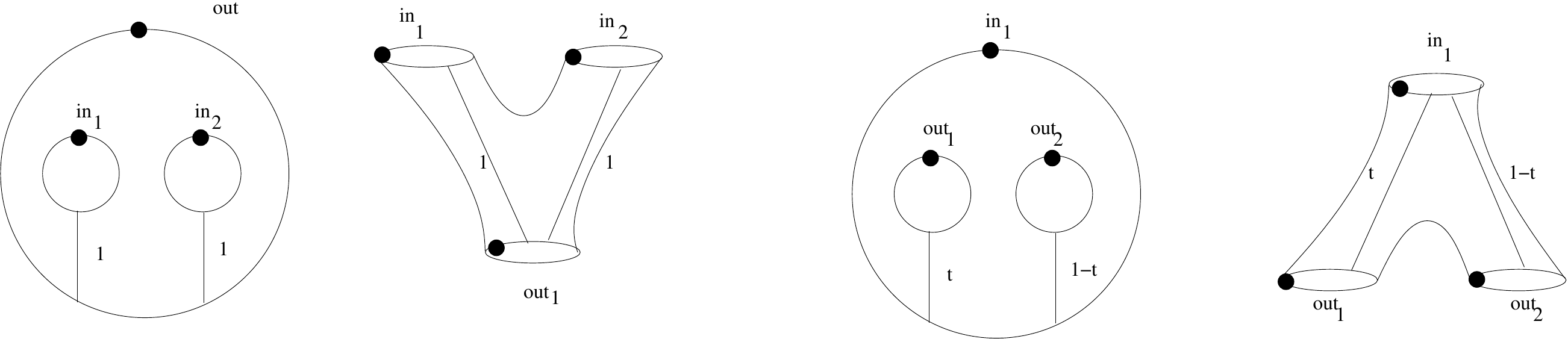} \hskip .2in
\includegraphics[align=c, width=.3\textwidth]{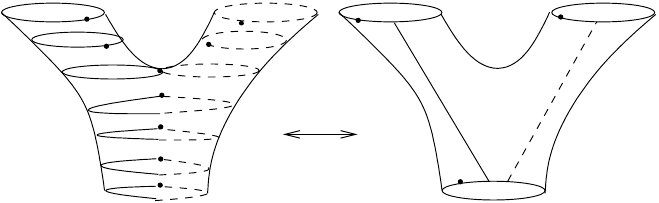}
\caption{\label{cellfig} The cell for the product (left) and coproduct (middle). The weights in the product case are both normalized to one, since each arc is incident to one input boundary. Asymmetrically, for the coproduct both arcs are incident to a single input boundary, so that only their combined weight is normalized to $1$, yielding a cell of dimension $1$. These operations are time reversal dual to each other. In the string picture (right) of \cite{woods} the two strings merge for the product moving down and moving up for the coproduct one string forms a figure 8 and breaks apart into two strings with the relative lengths $t$ and $1-t$. Each arc represents a piece of string}
\end{figure}

\subsection{The coproduct on $CH^*$}

\begin{thm}
\label{coprodthm}
Given a Frobenius algebra $A$ consider $CH:=CH^*(A, A)$.
The cell for the coproduct given in Figure \ref{cellfig}  acts, according to \cite[\S3.2.1]{hoch2}, as a coproduct morphism
\begin{equation}
\label{DeltaCHeq}
 \Delta_{CH}\in Hom(CH, CH^{\otimes 2})
 \end{equation}
  The formulas for its  non--zero components
$\Delta_{CH}(f)\in \bigoplus_{p+q=n-1} CH^p\otimes CH^q$ are explicitly given by
\begin{equation}
\begin{aligned}
\label{CHcoprodeq}
&\Delta_{CH}(f)[(a_1\odo a_p)\ot (a_{p+1}\odo a_{n-1})]\\
&= (-1)^p\sum_{C_1,C_2}\Cp_1 f(a_1\odo a_{p}\ot\Cp_2\Cpp_1 \ot a_{p+1}\odo a_{n-1})\ot \Cpp_2
\end{aligned}
\end{equation}

According to \cite[Theorem B]{hoch2} the boundary of this chain operation is given by the operation of the boundary of the cell, given in Figure \ref{boundaryfig}. It has two components and the operations corresponding to these are
\begin{equation}
 \del_0\Delta_{CH}:CH^n\to  CH^n\otimes CH^0 \mbox{ and } \del_1\Delta_{CH}: CH^n\to CH^0\otimes CH^n\nn
\end{equation}
which are given by the following explicit formulas: using $CH^0(A,A)=Hom(k, A)$, choosing $a_i\in A$ and $\lambda \in k$
\begin{equation}
\begin{aligned}
\label{boundareq}
&\del_0\Delta_{CH}(f)(\lambda\otimes (d_1\odo d_n))
=\lambda (1\otimes f(a_1,\dots,a_n)) \Delta(1)^2\\
&\del_1\Delta_{CH}(f)((a_1\odo a_n)\otimes \lambda)
=\lambda  \Delta(1)^2 (f(a_1,\dots,a_n)\ot 1)
\end{aligned}
\end{equation}

\label{constcor}
Furthermore,
the coproduct is well defined as a cohomology operation if  $\Delta(1)^2=0$. It is also a well defined cohomology operation
modulo the ``constant term'' $CH^0$ or relative to the constant term.

\end{thm}
\begin{proof}
The proof is in \S\ref{proof}, which also contains equivalent forms of the boundary operation involving $\D_A(f)$, see \eqref{bdopdfeq}.
\end{proof}
\subsubsection{Geometric/Gorenstein case}
\begin{lem}
\label{deltasqlem}
Let $A$ be graded Gorenstein---in particular this is the case if $A=H^*(X)$ for a connected Poincar\'e duality space $X$.
Then, $\Delta(1)^2=\eps(e)e_{\it top}\otimes e_{\it top}=(1\otimes e) \Delta(1)=\Delta_A(e)$ and TFAE (i) $\Delta(1)^2=0$, (ii)
 $\eps(e)=0$ and (iii) $e=0$.
\end{lem}
\begin{proof} If $A$ has socle in degree $d$, the total degree of $\Delta(1)^2$ is in degree $2d$ and this space is spanned by $e_{\it top}\ot e_{\it top}$.
It suffices to compute $(\eps\ot \eps)(\Delta_A(1)^2)=\la \Delta(1)^2,1\ot 1\ra=
\la \Delta_A(1), \Delta_A(1)\ra=\la \mu(\Delta(1)),1\ra=\eps(e)$. The other equation follow in similar fashion, using \eqref{deltaeq}. Since $e=\eps(e)e_{\it top}$, the equivalences follow.
\end{proof}

\begin{prop}
\label{factorprop}
If  $A$ graded Gorenstein, the action   corresponding to  the boundary components, which is  the boundary $\Delta_{CH}$, factors through maps to the degree $0$  part $A_0$ of $A$:
that is $CH^*(A,\bar A)\subset ker(\del_{0/1}(\D_{CH}))$ and
 \begin{equation}
 \label{imeq}
 \begin{aligned}
&Im(\del_0\Delta_{CH})\subset  CH^n(A,A_0)\otimes CH^0(A,A_0)\subset CH^n(A,A)\ot CH^0(A,A)\\
&Im(\del_1\Delta_{CH})\subset CH^0(A,A_0)\otimes CH^n(A,A_0)\subset CH^0(A,A)\otimes CH^n(A,A)
\end{aligned}
\end{equation}
\end{prop}

\begin{proof}
This follows from Corollary \ref{reducedcor}.
\end{proof}

Summing up:
\begin{thm}
\label{bdactionthm} If $A$ is graded Gorenstein,
the coproduct induces an operation on cohomology relative to or modulo the constants $CH^0(A,A_0)\simeq Hom(k,k)$ and thus is well defined on the reduced complex $\widetilde{CH}^*(A)$.
In particular, this is the case
for $A=H^*(X)$ for a connected Poincar\'e duality space $X$.

Furthermore, if  the Euler characteristic $\eps(e)=0$ vanishes,  the coproduct is  a cohomology operation on $HH^*(A,A)$ directly.

\qed
\end{thm}
\begin{rmk}
If $A$ is graded Gorenstein and $A=A_0 \oplus \bar A$. Since $\bar A=\bigoplus_{k>0} A_k$ is an $A$-$A$ bimodule, the constant maps, that is maps to $A_0=k$ can be identified with
$CH^*(A,A)/CH^*(A,\bar A)$.
\end{rmk}

\begin{rmk}
There is another way in which the constants in $CH^0$ appear in quotients.
If $A=1\oplus \bar A$ is augmented, then $CH_*(\bar A)$ which is linearly given by $\bar T\bar A$
computes the na\"ive Hochschild (co)homology of $\bar A$ \cite[\S1.4.3]{Loday}, and in the Gorenstein case the coproduct is well defined on this complex  and consequentially on its dual
as well.
\end{rmk}

\subsubsection{The coproduct as a homotopy}

The two boundary terms are  homotopic and  so are the operations.
In fact, the co-product {\em is} the homotopy between the left and right multiplication by elements of $CH^0$.
The boundary terms are also homotopic to the algebraic version of the pointwise coproduct of \cite{Ssigma}, see \S\ref{naivedualpar}.

Similar to the brace operation, we can regard the symmetrized coproduct $\D_{sym}:=\Delta_{CH}+\Delta_{CH}^{op}$.
Note that if one adopts signs as for the usual bracket, that is shifted degrees with the operation in the middle, see \cite[\S4.4]{del},
 the operation this is actually the anti--sym\-metrization of $\D$.

\begin{prop}
$\D_{sym}$  is also a well defined  homology operation.
It is null--homotopic modulo  $CH^0(A)$ or the constants in $CH^0(A)$
 in the case of $A$ being graded Gorenstein. Thus the coproduct is cocommutative modulo $CH^0(A)$ or
 in $\widetilde{CH}^*(A)$ is $A$ is graded Gorenstein.
\end{prop}
\begin{proof} See Example \ref{dopex}.
\end{proof}
Note, one does not really need that the algebra is connected. It could be the direct sum of connected (graded Gorenstein) components.
\subsection{Correlation functions and operations on  Hochschild (co)chains}
\label{copar}

\subsubsection{$Hom$ spaces and correlation functions}
\label{hompar}
The power
$\eta^{\otimes n}\otimes \check\eta^{\otimes m } $ is a  perfect pairing for $A^{\otimes n}\otimes \check A^{\otimes m}$.
For simplicity, we denote all these  by $\la \,,\,\ra$. Which precise form is used is determined by the type of elements the form is applied to. For example $\la a\otimes b, c\otimes d\ra=\la a,c\ra\la b,d\ra$.

Using the various dualites:
\begin{equation}
\label{dualizeFAeq}
Hom(A^{\otimes n}, A^{\otimes m})\simeq  A^{\otimes m}\ot \check A^{\otimes n}=\check A^{\otimes n+m}\simeq Hom(A^{\otimes n+m},k)
\end{equation}
Maps $Y\in Hom(A^{\ot n},k)$ are called correlations functions. Explicitly,
a correlation function $Y:A^{\otimes n}\to k$ defines an  elements in $Hom(A^{\otimes p},A^{\otimes q})$ for any $(p,q)$--shuffle
$\sigma$ via:
\begin{equation}
\label{dualizeeq}
\sign_{Z/2\Z}(\sigma)\sum_{C_1,\dots,C_q}   Y(\sigma(a_1 \odo a_p \otimes C^{(1)}_1 \odo C^{(1)}_q)) C^{(2)}_1 \odo C^{(2)}_q
\end{equation}
Where the sum is the multiple Sweedler sum for $q$ copies of the Casimir element and $\sign_{Z/2\Z}(\sigma)$ is the  Koszul sign for the shuffle.
These dualities extend to the tensor algebra   $Hom(TA^{\ot n}, TA^{\ot m})\simeq Hom(TA^{\ot n+m},k)$.
\begin{rmk}
By \eqref{dualizeeq},   $Y$ gives rise to different mophisms $\hat Y_{p,q}\in Hom(A^{\ot p}, A^{\ot q})$ for each $p+q=n$, which will be called {\em forms of $Y$}.
If $A$ is Frobenius then all these are equivalent. If it is not, some of these forms might exist apart from the others, see \S\ref{calcpar} for explicit examples
and in particular \S\ref{explicitpar} for the calculations relevant for the coproduct.
\end{rmk}

\subsubsection{Dualization to functions}
\label{CHpar}
\label{CHnota}

An element in $CH^n(A,M)$ is a sum of expressions $m\otimes  \check a_n \odo \check a_1$ with $m\in M$ and the $\check a_i\in \check A$.
For a Frobenius algebra,
$Hom(A^{\otimes  n},A)\simeq A \otimes \check A^{\otimes n}\simeq A^{\otimes n+1}$.  A function $f=a_0\ot \check a_n\odo \check a_1$ with $\check{a}_i=\la a_i, .  \ra$
is isomorphic to $\hat f= a_0 \odo a_n\in A^{\otimes n+1}\in TA$.
The first tensor factor plays a special role and will be called the module variable.   Vice--versa, $\hat f$ we recover $f$ as $f(b_1,\dots, b_n)=a_0\prod_{i=1}^n \la a_i,b_i\ra$.
The particular ordering is chosen to avoid an extra Koszul sign, cf. e.g.\  \cite{del}.

\subsubsection{Operations on $CH^*(A,A)$}
\label{conversionpar}
The dg--PROP action of \cite{hoch2} on $CH$ has homogenous components defined via  correlation functions whose definition proceeds as follows: Via the procedure given in \S\ref{cellpar}
a cell $c$ defines  correlation functions \eqref{cwtcor}, which are morphisms:
 \begin{equation}
\label{standardeq}
Y(c)_{p_i,q_j}
\in
Hom(A^{\otimes p_1+1}\odo A^{\otimes p_n+1}\otimes A^{\otimes q_1+1} \odo A^{\otimes q_m+1},k)
\end{equation}
where the $n$ and $m$ are part of the given data of $c$.
Dualizing the $A^{\ot q_i+1}$ according to \eqref{dualizeeq} one obtains a PROP action on $\bar TA$:
\begin{equation}
\hat Y(c)_{p_i,q_j}\in Hom(A^{\otimes p_1+1}\odo A^{\otimes p_n+1}, A^{\otimes q_1+1} \odo A^{\otimes q_m+1})
\end{equation}
Finally identifying the $A^{\ot k+1}\simeq CH^{k}(A,A)$ as in \S\ref{CHpar}, one obtains a dg--PROP action:
\begin{equation}
\label{opeq}
op_{CH}(c)_{p_i,q_j}\in \Hom(\bigotimes_{i=1}^n CH^{p_i},\bigotimes_{j=1}^m CH^{q_j})
\end{equation}
\begin{rmk}
By \S\ref{hompar} there are additional possible dualizations for the {\em individual operations}, see \S\ref{naivedualpar}.
For composable (PROPic) versions one needs to dualize these operations using Hochschild homology, see \S\ref{KWangpar}.
\end{rmk}

\subsection{Correlation functions and action on $CH^*$ from $c_\D$}
\label{CHcorpar}

The PROP cell for the coproduct 1-dimensional cell $c_{\Delta}$ is parametrized by an interval. The cell and its boundary $0$--cells are given in Figure \ref{boundaryfig}.
Notice that $\del_1C=\tau_{12}\del_0C$ where $\tau_{12}$ switches the ``out'' labels $1$ and $2$. Switching these two labels produces the cell for $\Delta^{op}$.

\subsubsection{The coproduct correlation function}
\label{coprodcorpar}
Using the procedure reviewed in \S\ref{cellpar} one duplicates arcs, assigns a local correlation function for each complementary region, and then takes the product of the local correlation functions to obtain the correlation function of the cell.
For the cell $c_\D$  one obtains one summand for each pair $(k,n)$  where the
left arc is duplicated $n$  and the right arc is duplicated $k-n$ times. The complementary regions are a central octagon and $n+k$ quadrilaterals.
This homogenous component  corresponds to a map $CH^{\ot k}\to CH^{\ot  n}\ot CH^{\ot k-n-1}$. The $(8,4)$ term is depicted in Figure \ref{coprodfig}.

\begin{figure}
\includegraphics[width=.6\textwidth]{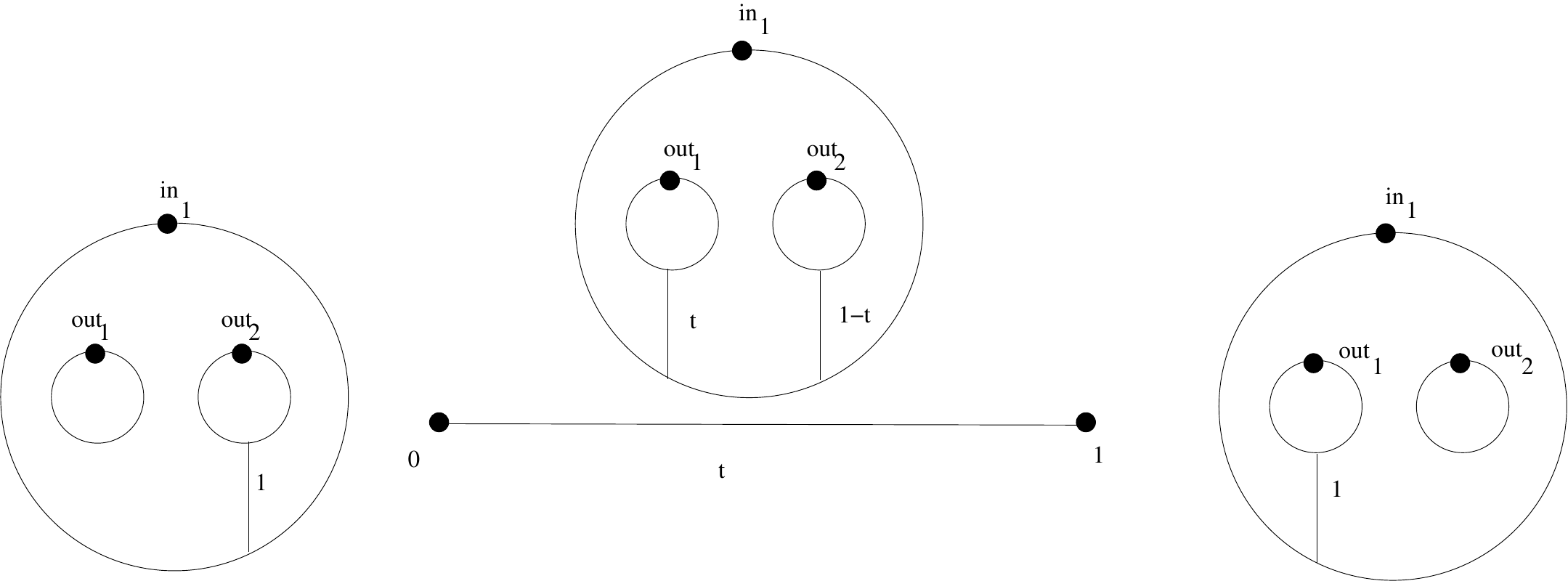}
\caption{\label{boundaryfig} The 1--dimensional cell and its two boundary points. }
\end{figure}
\begin{figure}
\includegraphics[width=.15\textwidth]{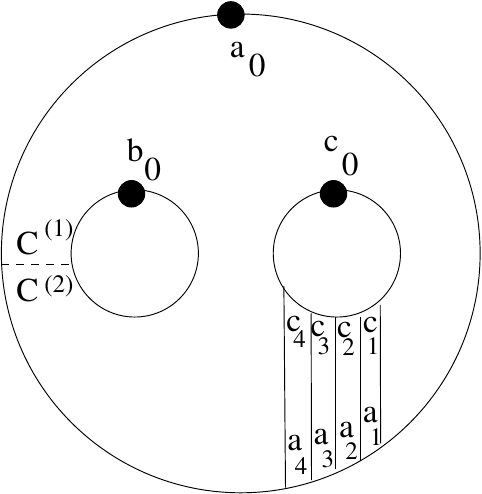}\hspace{1cm}
\includegraphics[width=.15\textwidth]{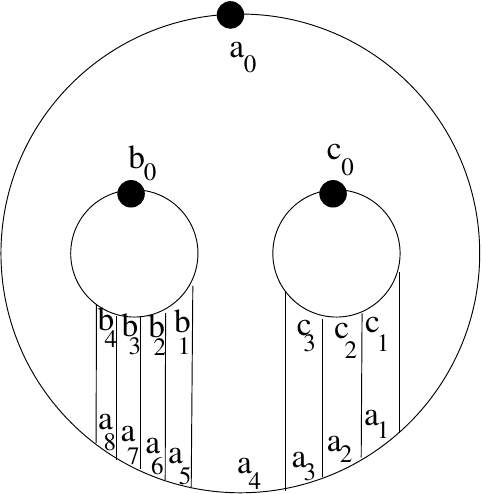}\hspace{1cm}
\includegraphics[width=.15\textwidth]{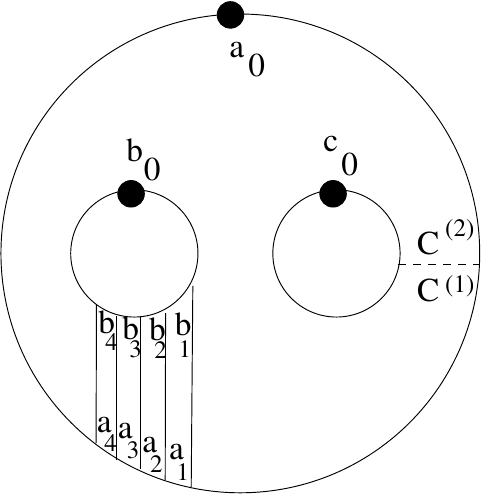}
\caption{\label{coprodfig}
\label{bdcoprodfig}
{Middle}: The $(8,3)$--summand of the component of the degree $1$ operation on $CH^*$ corresponding to $CH^8\to CH^3\otimes CH^4$. Cutting at the arcs yields one octagon $P_8$ and 7 quadrilaterals ($P_4$s). { Left:} the component  of the $\del_0$ boundary  operation $CH^4\to CH^4\ot CH^0$.
 {Right:}
The component  of the $\del_1$ boundary  operation $CH^4\to CH^0\otimes CH^4$. The extra cut for the annulus is the dotted line and decorated by
$C=\Delta(1)=C^{(1)}\ot C^{(2)}$. }
\end{figure}

\begin{prop}
\label{coprodcorprop}
The total correlation function for the cell $c_{\Delta}$ defining the degree $1$ coproduct $\Delta_{CH}$ is a product over the local correlation functions
$Y(c_{\Delta})=Y_A(P_8)\ot Y_{A} (P_4)\odo Y_{A}(P_4)\circ \sigma$, where the $Y(P_{2n})$ are given by \eqref{polygoncoreq} and $\sigma$ is a permutation.
The  formula of the operation on homogenous components  is given by
\begin{equation}
\label{Ycopeq}
\begin{aligned}
&Y(c_{\Delta})((a_0\odo a_{n}) \otimes(b_0\odo b_p) \otimes ( c_0\odo c_{n-p-1}))\\
&= \pm
\la a_0b_0a_{p+1}c_0 \ra \prod_{i=1}^p \la a_ic_i\ra \cdot \prod_{j=1}^{q}
\la a_{p+1+j}b_j  \ra
\end{aligned}
\end{equation}
\end{prop}
\begin{proof}
Decorating according  to \S\ref{standardpar},  the input pieces  of the boundary are decorated by $a_0\kdk a_n$ starting at the marked point going clockwise, that is in the opposite orientation of the boundary as it is an input. The two outputs are decorated by $b_0\kdk b_p$ and $c_0\kdk c_q$ respectively, also going  clockwise, which is  the induced orientation, used for outputs.
Cutting on the  arcs, one sees a central octagon whose sides are decorated by $(a_0,b_0,a_{p+1},c_0)$ in this cyclic order. The alternating sides of the quadrilaterals $P_4$
are decorated by $a_i,b_i$ on the left and by $a_{p+1+j},c_j$ on the right. The sign comes from the shuffle, shuffling the tensors into the given place according to \S\ref{cellorderpar}.
\end{proof}

\subsubsection{The  boundary correlation functions}
For the boundary of $c_\D$ the correlation function is more complicated as the complementary regions  are not simply polygons. The action according to \cite{hoch2}  is given by introducing in a system of extra cut--arcs decomposing each of the non--polygonal regions into polygons and decorating the two sides of the extra cut--arcs by Casimir elements. The procedure is reviewed in \S\ref{cellpar}, see Figure \ref{coprodfig} for the relevant example. In \S\ref{OTFTpar}, we prove different cut systems yield the same correlation function and show that the assumption of commutativity of $A$ made in \cite{hoch2} is unnecessary.
The  formula for the local correlation function is in \eqref{asscoreq}.
The appearance of the Casimir element is what makes the boundary factor through the constants.
After cutting these extra arcs, one is again left with a decorated polygon as above. In the case of the boundary of $c_\Delta$ one cut--arc suffices, see Figure \ref{bdcoprodfig}.

Decorating by elements of $A$, reading off the cyclic word, and integrating, one obtains the following operations:
\begin{prop}
 \label{boundarycorprop}
 Let $\del_0c_\Delta$ be the boundary at $t=0$ and $\del_1c_D$ the boundary at $t=1$,
then they define the following correlation functions:
\begin{equation}
\label{boundarycoreq}
\begin{split}
 Y(\del_0c_\D)(a_0\odo a_{n} \otimes b_0\ot  c_0\odo c_n )
 =\pm\sum \la a_0\Cp b_0\Cpp c_0\ra \prod_{i=1}^n\la a_ic_i\ra\\
 Y(\del_1c_\D)(a_0\odo a_{n} \otimes b_0\odo b_n \otimes c_0)
  = \pm\sum \la a_0b_0\Cp c_0\Cpp\ra  \prod_{i=1}^n\la a_ib_i \ra
\end{split}
\end{equation}
\end{prop}
\begin{proof}
Treating boundary at $t=1$,  there is only one arc to replicate.  The input is decorated by $a_0\kdk a_n$, the first output by $b_0\kdk b_m$, and the second output by $c_0$. After cutting, besides the $P_4$ quadrilaterals labelled by $a_i,b_i$, there is a central surface which is an annulus. One of the  two boundary components labelled by $a_0,b_0$ and the other by $c_0$. Inserting one cut arc and decorating it by the Casimir element yields the given correlation function according to \eqref{asscoreq}, see Figure \ref{bdcoprodfig}. The boundary at $t=0$ is analogous, albeit that the sole arc now runs to the other input.
\end{proof}
\subsubsection{Proof of Theorem \ref{coprodthm}}
\label{proof}
 Let $a_0\odo a_n$ be the input tensor and $b_0\odo b_p$ and $c_0\odo c_{n-p-1}$ be the two output tensors. Using the calculations for the dualization of $\int_2$ and $\int_4$ given in \S\ref{explicitpar}, \eqref{int2} and \eqref{int4}, and summing up the contributions:
\begin{equation}
\label{coprodhatYeq}
\begin{aligned}
&\hat Y(c_\D)(a_0\odo a_n)
=\sum_{p =0}^n \sum \pm (a_0^{(1)}\ot a_1\odo a_{p})\otimes (a_0^{(2)}a_{p+1}\ot a_{p+2} \odo a_{n})\\
&=\sum_{p =0}^n \sum \pm (\Cp a_0\ot a_1\odo a_{p})\otimes (\Cpp a_{p+1}\ot a_{p+2} \odo a_{n})
\end{aligned}
\end{equation}
where $\pm$ is the sign coming from shuffling in $a_0^{(2)}$.

Identifying the tensors with elements of $CH^*$, this translates $\D_{CH}:=op(c_\D)$ according to \S\ref{CHpar}.
For $f\in CH^n$:
$\Delta_{CH}(f)\in \bigoplus_{p+q=n-1} CH^p\otimes CH^q$ is given by \eqref{CHcoprodeq}
as $(\Cpp_1 \check a_p)(\Cp_2)=\la \Cpp_1  a_{p+1},\Cpp_1\ra=\la  a_{p+1},\Cp_2 \Cpp_1\ra=\check a_{p+1}(\Cp_2\Cpp_1)$.

For the boundary  the calculation for $\int_5$ in \S\ref{explicitpar} \eqref{int5} results in:
\begin{equation}
\label{bdeq}
\hat Y(\del_1c_\D)(a_0\odo a_n)= \sum (C^{(1)} a_0^{(1)}
\otimes a_1\odo a_n) \otimes C^{(2)} a_0^{(2)}
\end{equation}
where we used Sweedler's notation.
This in turn yields the operation
\begin{equation}
\label{bdopdfeq}
\begin{split}
\del_1\Delta_{CH}(f)((d_1\odo d_n)\otimes \lambda)
=(-1)^n \sum  \lambda [ (\Cp\otimes \Cpp) \Delta_A(f(d_1,\dots, d_n))]\\
= \lambda \Delta(1)\Delta_A(f(d_1,\dots,d_n))=\lambda \Delta(1)^2 ( f(d_1,\dots,d_n)\ot 1)
\end{split}
\end{equation}
and similarly for $\del_0\Delta_{CH}$. Where the last equality comes from \eqref{deltaeq}, viz.\
$
 \sum(\Cp\otimes \Cpp)\Delta_A(a)=\D_A(1)\D_A(a)=\Delta(1)^2
(a\otimes 1)$
\qed
\begin{cor}
\label{reducedcor}
If $A$ is graded Gorenstein, then the boundary correlation functions vanish  unless $a_0,b_0,c_0\in A_0$.
Dually,
$\del_{0/1}\D_{CH}(f)=0$ unless $f:A^{\ot n+1}\to A_0\simeq k$ is a constant map and the image of $\del_{0/1}\D_{CH}(f)$
has image as sepcified in \eqref{imeq}.

\end{cor}

\begin{proof}
Let $A$ have socle in dimension $d$, we see that each term $a_0\Cp b_0\Cpp c_0$ has degree at least $d$ and hence all the terms of the correlation functions are
$0$ unless $a_0, b_0$ and $c_0$ are of degree $0$ and hence all multiples of the unit $1\in A_0$.

In particular, the condition that
$a_0\in A_0$ implies that on $CH$
the operation is zero on any   map $f$ not having $A_0$ as image, and  $b_0,c_0\in A_0$ implies that the output functions are also maps to $A_0$.
\end{proof}

\begin{cor} If $A$ is graded Gorenstein,
the coproduct is a well defined cohomology operation,  in the complex $\overline{CH}^*(A,\bar A)$. \qed
\end{cor}

\subsection{Generalizing the actions}
\label{gendisccopar}
Using the point of view of \S\ref{viewpar}, the operations generalize from a Frobenius algebra in several ways.
 As the $\int_2$ terms represent identity morphisms, see \S\ref{explicitpar}\eqref{int2}, the term $\int_4$ is the only interesting one  in  $\hat Y(c_\D)$.
 In the form presented in \eqref{coprodhatYeq},  the module variable $a_0$ needs to be of the type $\CA_A$ with the rest of the
tensor variables lying in $A$, see \S\ref{explicitpar}\eqref{int4}. This means that the equation if well defined as a morphism $CH_*(A,\bisub{A}{\CA}{A})\to CH_*(A,\bisub{A}{\CA}{A})\ot CH_*(A,\bisub{A}{\CA}{A})$, for instance on $CH_*(C^*(M),C_*(M))$.

So, if $\check A$ is a coalgebra, for instance if $A$ is finite dimensional of of finite type,
then the operation exists as a morphism $CH_*(A,\check A)\to CH_*(A,\check A)$, where $A$ only needs to be associative.
 In case that one cannot identify $(A\ot A)^{\vee}$ with $\check A\ot \check A$, then specifying a special element $C$, see \S\ref{explicitpar}\eqref{int2}
allows one to use formalism of operadic correlation functions \cite[\S2]{hoch2}. Such an element is available if the coalgebra $\bisub{A}{\CA}{A}$ is pointed in Quillen's sense.

Using the alternative form, \eqref{homcoefeq}, one obtains a map $CH^*(A,A)\to CH^*(A,\bisub{A}{\CA}{A})\ot CH^*(A,\bisub{A}{\CA}{A}))$, thus in particular
$CH^*(C^*(X),C^*(X))\to CH^*(C^*(X),C_*(X))\ot CH^*(C^*(X),C_*(X))$. If $X=M$ is a manifold using the isomorphisms \eqref{cocoeffisoeq} and \eqref{hocoeffisoeq} one obtains a map
$H_{*+d}(LM)\to H_*(LM)\ot H_*(LM)$.

For the boundary operations, the discussion is analogous using \S\ref{explicitpar}\eqref{int5}, but in any formulation, due to the cut, there is the need for the special element $C$.
\section{Actions on (Co)chains of loop spaces and their geometric interpretation}
\label{looppar}
\subsection{Manifolds, Poincar\'e duality and intersection}
\label{intpar}
Let $M$ be a compact oriented connected manifold, then $A=H^*(M,k)$ is a Frobenius algebra over $K$ with $\mu_A=\cup$, $\eps=\int_M$ is the cap product with the fundamental class of $[M]$ followed by the augmentation map.
The duality between $A$ and $\check A=H_*(M,k)$ is known as Poincar\'e duality.

The integral $\int ab$ has the following dual geometric interpretation. Let $\check a$ and $\check b$ be Poincar\'e dual cycles intersecting transversally then
$\int ab$ is zero unless $\check a$ and $\check b$ have complementary dimensions and then $\int ab=\#$ of intersection points $a\pitchfork b$.

\subsection{Loop space models using Hochschild (co)chains}
\subsubsection{Geometric motivation}
If we regard a singular chain $c$ on the free loop space $LM=C(S^1,M)$,
we get a  chain $b_*(c)$ in $M$ by the push--forward with respect to the base--point map $b:LM\to M$ which sends $\phi$ to $\phi(0)$.
Evaluating at different points of $S^1$ gives similar maps.

The algebraic structure of Hochschild cochains
is given by sampling $S^1$ by  sequences of $n+1$ points that are cyclically ordered and coherent ---the first point always being $0$.
The $n+1$ points yield  $n+1$ singular chains. This gives a sequence  parameterized by $n$.
The $i$--th point may collide with the $i+1$st point lowering the point count in which case one should obtain the family with less points.
This is the coherence. Both points $1$ and $n$ can collide with $0$ which gives the extra degeneracy.

In terms of elements of $A^{\otimes n+1}$ the element $a_0\odo a_n$ represents the dual homology classes swept out by the $n+1$ points, that is $a_i$ is dual to the
homology cycle swept out by the $i$-th point and $a_0$ is dual to the base points of the loop.

\subsubsection{Cosimplicial viewpoint according Jones/Cohen-Jones \cite{jones,CJ}}

The sampling is formalized as follows:
 using a simplicial structure $S^1_\bullet$ on $S^1$ one obtains a cosimplicial structure on $Hom(S^1_\bullet,X)$ whose totalization gives
back the loop space. In fact, $Hom(S^1_\bullet, X)$  is cocyclic since $S^1_\bullet$ is cyclic. This cyclic structure is the reason for the existence of the BV operator.
More precisely, one has maps
\begin{eqnarray}
f_k:\Delta^k\times LX&\to& X^{k+1}\nn\\
(0\leq t_1 \dots \leq t_k\leq 1)&\mapsto& (\gamma(0),\gamma(t_1),\dots,\gamma(t_k))
\end{eqnarray}
which one can think of discretizing the loop.
These maps dualize to
\begin{eqnarray}
\bar f_k:LX&\to&Hom(\Delta^k,X^{k+1}) \nn\\
\gamma&\mapsto&(t_0\dots, t_k)\mapsto (\gamma(t_0),\dots,\gamma(t_k))
\end{eqnarray}
which are compatible with coface and codegeneracy maps; see \S\ref{simppar} for details.

\begin{thm}\cite{CJ,jones}
Let $X$ be a space and $f: LX\to \prod_{k\geq 0}Map(\Delta^k,X^k)$ be the product of the maps $\bar f_k$ then $f$ is a homeomorphism onto its image. The image is the subspace
$Tot(Map(S^1_\bullet,X))$, whose elements are those sequences that commute with the coface and codegeneracy maps.
\end{thm}

A singular $l$--chain $c_l:\Delta^l\to LX$ can be regarded as a family of
loops  $\gamma_\bt$ depending on ${\bf t}\in \Delta^l$. Its discretization gives a family of
maps $\Delta^l \times \Delta^k\to X^{k+1}$ which is an $l+k$ chain on
 $X^{k+1}$. The chain is given by the usual shuffle product formula which expresses the bi--simplicial $\Delta^{l}\times \Delta^k$ as a union of simplices.

Thus pulling back along the $f_k$ and using the Alexander Whitney map $AW:C_*(X^{k+1})\to C_*(X)^{\otimes k+1}$, one obtains maps
\begin{equation}
f_k^*:C^*(X)^{\otimes k+1}\to C^{*-k}(LX)
\end{equation}
\begin{thm}\cite{jones,CJ}
\label{CJthm}
The homomorphisms $f_k^*$ define a chain map
\begin{equation}
f^*:CH_*(C^*(X))\to C^*(LX)
\end{equation}
which is a chain homotopy equivalence when $X$ is simply connected. Hence it induces an isomorphism
\begin{equation}
f^*:HH_*(C^*(X))\stackrel{\simeq}{\to} H^*(LX)
\end{equation}
dualizing these maps  and using that $HH_*(C^*(X))=HH_*(C^*(X);C^*(X))$ yields
\begin{equation}
f_*:C_*(LX)\to CH^*(C^*(X);C_*(X))
\end{equation}
which is a chain homotopy equivalence when $X$ is simply connected. Hence it induces an isomorphism
\begin{equation}
\label{hocoeffisoeq}
f_*:H_*(LX)\stackrel{\simeq}{\to}HH^*(C^*(X),C_*(X))
\end{equation}
\end{thm}

\begin{rmk}
The direct dualization yields the dual of the complex $Hom(CH_*(C^*(X)),k)$.
 If $A$ is finite dimensional or of finite type, then as remarked previously, up to signs $CH^*(A)=Hom(CH_*(A),k)\simeq CH^*(A,\check{A})$\cite[1.1.5]{Loday},
with the isomorphism given by $F\leftrightarrow f$ as defined by
\begin{equation}
F(a_0,\dots, a_n)=f(a_1,\dots,a_n)(a_0)
\end{equation}
Taking this as a definition always gives a map $ CH^*(A,\check{A})\to CH^*(A)$.
 In total, the map $f_*$ can be seen as a the map that takes
an $l$ dimensional family of loops $\gamma_\bt$ to the evaluation maps
\begin{equation}
F_k=ev_{AW(\gamma_\bt(0),\dots, \gamma_\bt(t_k))}
\in Hom(CH^*(C_*(X)),k)
\end{equation}
where on the right hand side the  degree is $k+1$ and the total degree  is $k+l$.
This gives the explicit description with homological coefficients.
\end{rmk}

Using the same kind of rationale Cohen-Jones also prove  a second description with cohomological coefficients.
\begin{thm}[Corollary 11, Theorem 1 \cite{CJ}]
For any closed (simply connected) $d$ dimensional manifold $M$:
$H_{*+d}(LM)\simeq H_*(LM^{-TM})$.
And, there are naturally defined chain maps $f_{k,*}$ which fit together to
define a chain homotopy equivalence
\begin{equation}
f_*: C_*(LM^{-TM})\to CH^*(C^*(M),C^*(M))
\end{equation}
inducing an isomorphism
\begin{equation}
\label{cocoeffisoeq}
f_*: H_*(LM^{-TM})\simeq HH^*(C^*(M),C^*(M)) \simeq H_{*+d}(LM)
\end{equation}
\end{thm}

\subsubsection{Discretizing and Dualizing}
\label{simppar}
We give the explicit (co)-face and (co)-degeneracy maps of the simplicial/cosimplicial structures at the various level. This allows us to identify the constant loops in the Hochschild cochain complex.
They may also be used to find the Hochschild cochain representations of families of loops used in the arguments of string topology \cite{CS,GH} using the totalization.

For a discretized loop $\gamma$  these are:\\
\begin{tabular}{l|l}
$\delta_i:\Delta^k\to \Delta^{k+1}$&
$Hom(\Delta^k,X^{k+1})\to Hom(\Delta^{k+1},X^{k+2})$\\
$(\dots,t_i,\dots)\mapsto (\dots, t_i,t_i,\dots)$&$\gamma(\dots,t_i,\dots)\mapsto \gamma(\dots,t_i,t_i,\dots)$\\
$\sigma_i:\Delta^{k+1}\to \Delta^k$&$Hom(\Delta^{k+1},X^{k+2})\to Hom(\Delta^{k},X^{k+1})$\\
$(\dots,t_i,\dots)\mapsto (\dots, \hat t_i,\dots)$&$\gamma(\dots,t_i,\dots)\mapsto \gamma(\dots,\hat t_i,\dots)$\\
\end{tabular}\\
Thus, the map $\delta_i$ induces the map $\Delta_{i,*}$ which after applying the AW map $C_*(X^k)\to C_*(X^{k+1})$ is just the coproduct.

For families/homology classes using the diagonal maps $\D_i$ which repeat the $i$--the entry and projection $\pi_i$ which omit the $i$--entry:\\
\begin{tabular}{l|l}
$\Delta_i:X^k\to X^{k+1}$&$\Delta_{i,*}:
C_*(X^k)\to C_*(X^{k+1})$\\
$(\dots,\gamma_i,\dots)\mapsto (\dots, \gamma(t_i),\gamma(t_i),\dots)$&$\gamma_{\bt}(\dots,t_i,\dots)\mapsto \gamma_\bt(\dots,t_i,t_i,\dots)$\\
\multicolumn{2}{r}{$\D_i:=id\odo id\otimes \Delta\otimes id \odo id:C_*(X)^{\otimes k}\to  C_*(X)^{\otimes k+1}$}\\
&$\gamma_0\odo\gamma_k\to\gamma_0\odo \gamma_i^{(1)}\otimes \gamma_i^{(2)}\odo \gamma_k$\\
$\pi_i:X^{k+1}\to X^k$&$C_*(X^{k+1})\to C_*(X^{k})$\\
$(\dots,\gamma(t_i),\dots)\mapsto (\dots, \hat \gamma(t_i),\dots)$&$\gamma_\bt(\dots,t_i,\dots)\mapsto \gamma_\bt(\dots,\hat t_i,\dots)$\\
\multicolumn{2}{r}{$\eps_i:=id\odo id \otimes \eps\otimes id\odo id:C_*(X)^{\otimes k+1}\to  C_*(X)^{\otimes k}$}\\
&$\gamma_0\odo\gamma_k\to \gamma_0\odo \eps(\gamma_i)\odo \gamma_k$\\
\end{tabular}

Finally dualizing, in the manifold setting, we see that these morphisms go to
\begin{eqnarray*}
\mu_i:=id\odo \otimes \mu \otimes id\odo id&:& C^*(M)^{\otimes k+1}\to C^*(M)^{\otimes k}\\
\eta_i:= id \odo id\otimes \eta\otimes id \odo id&:&C^*(M)^{\otimes k}\to C^*(M)^{\otimes k+1}
\end{eqnarray*}
where $\mu$ is the multiplication given by the $\cup$ product and
$\eta:\mathbb{Z}\to C^*(M)$ is the unit.

\subsubsection{Constant loops}
The discretized series for a constant loop $\gamma(t)\equiv x \in M$ is given using the maps $\delta_i$
$$(\gamma(0),\dots, \gamma(t_k))=\delta_{k-1}\cdots\delta_0\gamma(0)$$
Thus a constant family of loops has the series
$$
AW(\gamma_0\odo\gamma_k)=\Delta_{k-1}\circ\cdots \circ \Delta(\gamma_0)
$$
which can be reconstructed from
$$
\gamma_0=\eps_1 \circ \cdots \circ \eps_1 (\gamma_0\odo\gamma_k)
$$
Dually the cochain/cohomology sequence is given by
$$
(\check{\gamma}_0\odo \check{\gamma}_k)=AW^*(\eta_k \circ\cdots \circ\eta_1(\check\gamma_0))
$$
From these formulas one obtains that evaluation at a constant loop in degrees bigger than $0$ is in the degenerate subcomplex and these do not appear in the  normalized complex.

\begin{prop}
\label{constloopprop}
In higher degrees, the image of constant loops is  in the degenerate subcomplex. In the normalized complex their image is  $\overline{CH}^0(C^*(M),$ $C_*(M))=C_*(M)$.
Moreover, choosing a base point for $M$ defines a constant loop as a base point for $LM$ and the reduced homology of $\bar H_*(LM)$ is quasi isomorphic to the dual of the reduced
  chain complexes $\widetilde{HH}_*(C^*(M))$ and $\widetilde{HH}^*(C^*(M))$ computed by the reduced (co)chain complexes.
\end{prop}
\begin{proof}
The only thing left to prove is the surjectivity. For this one identifies a singular chain $f:\D^k\to M$ as a family of constant loops.
\end{proof}

\subsection{Geometric interpretation for loop spaces}
\label{interpretpar}

The preceding theorems and corollaries translate the algebraic results to a geometric  interpretation in terms of loops. By this we mean  that the given algebraic operations
reflect a geometric situation, in which usually transversality is assumed.  This is analogous to the discussion of transversal intersection in \S\ref{intpar} and  quantum cohomology \cite{maninbook}, where the true operations are the Gromov-Witten invariants and the geometry they reflect is the enumerative geometry, which is itself elusive.

For the loop space geometry this agreement with the geometry that  applying discretization given via the totalization to a
 geometric input family, e.g.\ constant loops or figure 8 loops, is commensurate with the algebraic operations.

\subsubsection{Figure 8 loops}
\label{fig8par}
We define the subspace of Figure 8 loops $F_8\subset Tot(Map(S^1_\bullet,X))$, those maps that factor through $Tot(Map(S^1_\bullet \vee S^1_\bullet,X))$ for a given simplicial model of the map $S^1_\bullet \to S^1_\bullet \vee S^1_\bullet$.
These can be represented by sequences $(\g(0),\g(t_1)\kdk \g (t_k))$ for which $\g(t_i)=\g(0)$ for some $i$. That is they are in the image of the small diagonal map $\D_{0,i}:X^{k+1}\to X^{k+2}$ duplicating the first and $i+1$st factors.
Decomposing $\D^k=\D^{k_1}*\D^{k-k_1-1}$ induces maps $Hom(\D^k, X^{k+1})\to Hom (\D^{k_1},X^{k_1+1})\times Hom(\D^{k-k_1-1},X^{k-k_1-1})$. When restricted to $F_8$ these maps yield coherent families and yield a map
$F_8\to Tot(S^1_\bullet, X)\times Tot(S^1_\bullet, X)$.  Let  $L_8M\subset LM$ be the space of these loops and $\Delta_8:L_8M\to LM\times_M  LM\subset LM\times LM$ the map constructed above via the totalization.
 That is we obtain models for the maps $LM\stackrel{i_8}{\leftarrow}F_8\stackrel{\Delta_8}{\to} LM\times_M LM \subset LM\times LM$,
 where  $i_8:F_8\to Tot(S^1_\bullet, X)\approx LM$.

 \subsubsection{Levels of action}

By Proposition \ref{constloopprop}, one can identify the
constant loops with $\overline{CH}^0$ in the normalized chain complex and hence Corollary \ref{constcor} tells us that the operations are  well defined modulo constant loops as in \cite{GH}.
We even have more, namely that the coproduct already descends to operations relative to a base point constant loop.
The three  levels of \S\ref{threepar} as they relate to loop spaces are:

(1)  Since $E^1(CH^*(C^*(M)),C_*(M))=CH^*(H^*(M),H_*(M))\simeq CH^*(H^*(M))$, an action on this page is exactly the case discussed above for the action on $CH^*(A,A)$ for the Frobenius algebra $A=H^*(M)$.

(2) Note that in the formulas $\eps=\int$ lifts to the chain level as it  can be replaced by capping with the fundamental class $\cap[M]$. The multiplication $\cup$--product also lifts the chain level. This means that the correlation functions all lift to $CH^*(C^*(M))=CH^*(C^*(M),C_*(M))$ $\simeq CH^*(C^*(M),C^*(M))$, which means they are well defined and induce the operations on the $E^1$ page.
To obtain  PROP action one has to ``dualize'' the outputs.
This can be done by {\em  choosing a cochain representative} of the  diagonal $C=\sum \Cp\otimes \Cpp$ and simply using \eqref{dualizeeq} as a definition.

The formalism of using a propagator $C$ to define actions  is discussed in detail is \cite[\S2]{hoch2} under the name of operadic correlation functions. The relevant result is  \cite[Theorem 4.15]{hoch2}.

(3) Lastly, if one does not look at the whole PROP of operations on one space, one can pick individual operations and see if picking cleverly from the  descriptions
$CH^*(C^*(M),C_*(M))$, $CH^*(C^*(M), C^*(M))$ or $CH_*(C^*(M))$ for $C_*(LM)$ yields a formula that does not utilize dualization.

The classical example is the  product. In this case,  one can take the coefficient module to be an algebra. This was the motivation for \cite{CJ}. If fact it is clear from our formulas, that the whole little discs suboperad will act when picking cohomological coefficients \cite{del}.
For the BV action, the natural space is $\overline{CH}^*(A)$, for a Frobenius algebra $A$  together with its cyclic structure \cite{cyclic}.
For the coproduct the natural morphism is    $CH_*(C^*(M),C_*(M))\to CH^*(CH^*(M),CH_*(M))$ as now the coefficients have a coproduct structure, see \S\ref{gendisccopar}.

\subsubsection{Coproduct on loop space}
\label{geocoprodpar}
We will now discuss the degree $1$ coproduct from all three different points of view.

(1) In terms of dual classes, we see that
 the first term says that $a_0,b_0,a_{p+1}$ and $c_0$
 ``coincide'' in the sense that if we
 use the interpretation of the $\cup$ product as intersection of the dual homology chains, see \S\ref{intpar}. The degree count says says that the loci need to intersect in points (counted with multiplicity).

  This means that all the base points and the $p+1$st point coincide, which is indeed the situation of \cite{GH,Sopenclosed,Ssigma,Sullivanoverview}. Summing over all $p$  re--parameterizes the loop.
 The map  sends the first loop which is a figure 8, to the two loops as in the Figure \ref{figure8}.

 \begin{figure}
 \includegraphics[width=.7\textwidth]{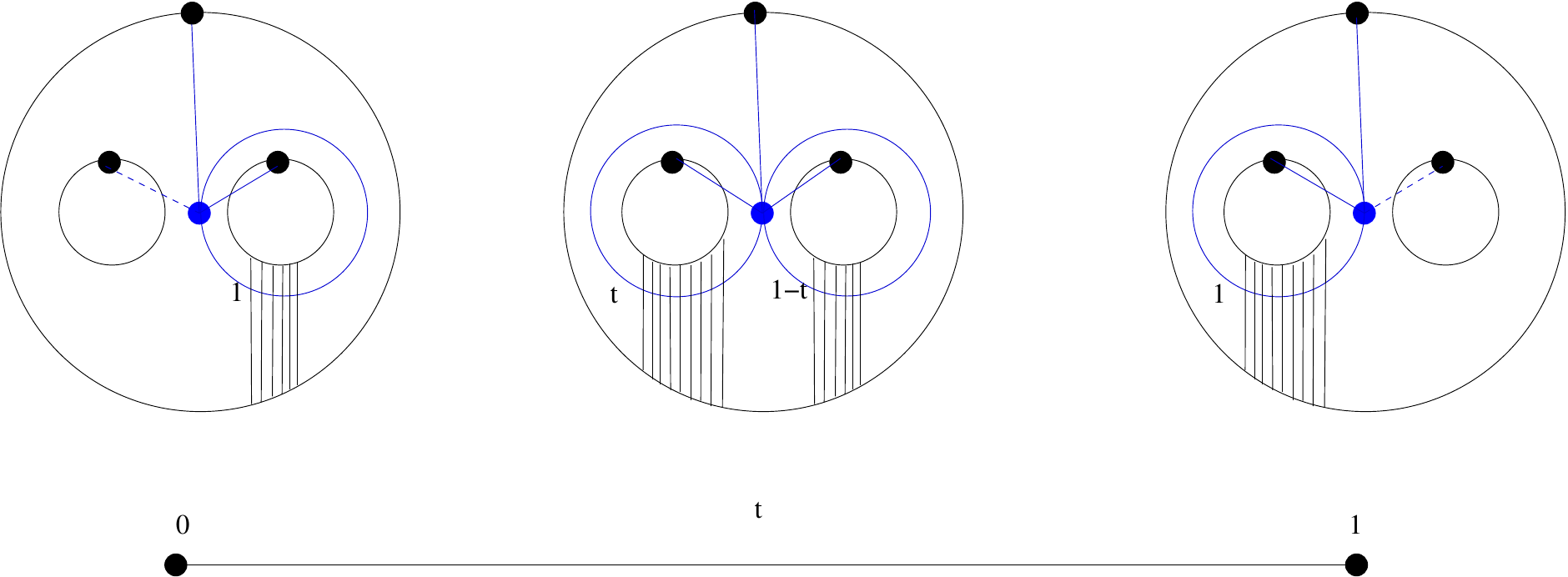}
 \caption{\label{figure8} The blue graph is the dual graph to the weighted arc system. The tails and dotted tail  keep track of the almost ribbon structure and the base points \cite[Appendix A1]{postnikov}.}
 \end{figure}

(2) Lifting to chains, we see from \eqref{coprodhatYeq},  that he coincidence conditions for spawning off of a the loop via the coproduct
are being forced by the intersection with the diagonal ---again forcing the situation of \cite{GH} that encodes \cite{Sopenclosed,Ssigma,Sullivanoverview}.

(3) As discussed in \S\ref{gendisccopar} the operations lift to  operations  $CH^*(C^*(M),C^*(M))\to  CH^*(C^*(M),C_*(M))\ot  CH^*(C^*(M),C_*(M))$ which also give a chain model for  the loop space homology.
Interpreting this as homology classes given by discretezations of loops, and uses $C\in C_*(M)\ot C_*(M)$
as
a chain representative of the diagonal the terms
 $\Cp a_0$ and  $\Cpp a_{p+1}$ becomes  the intersections $C \pitchfork (a_0\ot a_{p+1})$.
Restricting to the space where there is such  intersections is the starting point of \cite{GH}.

\subsubsection{Boundary operations}
\label{bdgeopar}
We again have the three points of view above:

(1) On the level of classes and dual intersections, we see that \eqref{bdeq} says that the loop itself is left alone and spawns off a second constant loop at its base point.
We furthermore see that due to degree reasons $a_0,b_0,c_0$ all must be of degree $0$. This is due to the fact that the coproduct and the intersection with the diagonal produce a term $\Delta(1)$ which is already in top degree. This means that dually $a_0$ and hence $b_0$ and $c_0$ which coincide up to scalars have to sweep out the entire $M$. The dual interpretation is consistent with the intersection interpretation. Indeed, just like the cup product with $1$ is trivial, so is the intersection with all of $M$.
The loop that spawned off is a constant loop.

(2) The lift to chains is possible along the same lines as in the coproduct case and the geometric statements are those made above.

(3)  The interpretation of $C$ as a chain representative of the diagonal  intersected with the relevant homology classes applies as in the coproduct case.

 \subsubsection{Identifying  coproducts}
An algebraic realization of the loop space is given by a Frobenius dgA model for $M$. This exists for instance if $M$ is formal. Such a model has also more generally been provided   in \cite{LambStan}.

A transversal realization of the string topology operations is a geometric constructions which on transversal families of loops induces the \cite{CS} type string topology operations.
The \cite{GH} coproduct is of this type.
Transversal realization is also  the input for Umkehr maps \cite{Umkehr} and guarantees a Cohen-Jones \cite{CJ}  type of setup as postulated in \cite[\S 4.6]{KLP}.
Umkehr on (co)homology uses Poincar\'e duality \cite{Umkehr} and as in \S\ref{intpar} turns intersection into cup products.
The  map should be
$\Delta^!_{*} i_{8!}:H_*(LM)\to H_*(LM)\ot H_*(LM)$where $i_{8}$ is a Umkehr map, that means a map going the ``wrong way''.  This and other geometric schemes
  this can be  traced through the discetization and starting at $E^1$ the transversal intersection of loops can hence  be characterized via the previous calculations.

\begin{cor}
\label{constloopcor}
 For an algebraic realization of the coproduct or a  transversal realization the coproduct descends to a cohomology operation on the reduced complex $\widetilde{HH}^*(C^*(M))$, inducing a coproduct on $\bar H_*(LM)$.
 Such a realization induces a morphism on the $E^1$ page of the spectral sequence which is given by the $\D_{CH}$.
\end{cor}
\begin{proof}
Following through the discretization as detailed above, we see that the formula for the coproduct is indeed the transversal intersection in the form of cup products.
\end{proof}

\section{Geometry \& actions of CW complexes and dualties}
\label{cellpar}

\subsection{Cells}
The correlation functions of \cite{hoch2} are given for cells in a CW complex $\A$ together with an interval marking via
 data indexing the cell.
The  complex $\A$ is a CW complex whose cells indexed by classes represented by an oriented surface $\Sigma$, with enumerated boundary components $\del\Sigma=\amalg_{i=1}^n S^1$,
 one marked point $p_i$ in each boundary component, and   an  arc system.  An arc system is a set  $\a$  of nonintersecting embedded curves, aka.\ arcs, that run from boundary to boundary not hitting the marked points which are not parallel and not parallel to the boundary. This configuration is considered up to isotopy and mapping class group action.  The classes $c=[(\Sigma, p_i, \a)]$ index the cells of a CW complex $\A$. The dimension of a cell is $|\a|-1$ and the interior of a cell naturally identified with the open simplex $\dot \Delta^{|a|-1}\subset \R^{|\a|}$.  The attaching maps or equivalently the cell boundaries are given by removing arcs. This is induced by a simplicial differential for which  all arcs are enumerated first according to the boundary components and then according to their order on the boundary, cf \cite{KLP,hoch1} for more details.

 \subsubsection{Discretization}
 An integer weight for a set of arcs $\a$   is given  a map $\wt:\a\to \N_{>0}$. A discretized cell  is given by $c_{disc}=[(c,\wt)]$. As an arc system, this is represented by replacing each arc $a$ by $\wt(a)$ parallel arcs.
 The differential is again given by removing arcs, which now is a sum of lowering the degree $\wt$ of the arcs in $\a$ by $1$ and removing the arc if the resulting weight is $0$ see \cite{hoch2} for details.

 \subsubsection{Interval/angle marking and action}

An interval is the part of the boundary between two arcs. In \cite{hoch2} the intervals are called angels as they are the angles of the arc graph \cite[1.1.2]{hoch2}.
A marking is a morphisms $\mk:interval\to \{0,1\}$ with the condition that each interval that contains a marked point, also called the module-interval, has value $1$.
 Intervals between parallel arcs are called splitting intervals. For simplicity we will restrict the value of $\mk$ to be $1$  on splitting intervals.
 The other intervals are called inner intervals and the function $\mk$ is completely determined by its value on these.
 The data  $(c,\wt,\mk)$ defines a homogeneous correlation function \eqref{cwtcor}. The correlation functions for a marked cell $(c,\mk)$ is given by summing over all possible weights \eqref{localsumeq},
 with the marking being the one fixed by its value on inner intervals.

Intervals with value $1$ will be referred to as  marked or active and the intervals with value $0$ as unmarked or inactive. This terminology
 avoids a possible confusion as marking by $0$ means decorating by the unit $1\in A$ for the correlation functions.

\subsection{Correlation functions for a cell}
\label{cellcorpar}
\subsubsection{Local OTFT correctors}

Let $S$ be a surface with enumerated  boundary components $\del F=\amalg_{j=i1}^b S^1$ and $d_i$ marked intervals on each boundary component.
An OTFT based on a Frobenius algebra $A$ assigns a correlation function
\begin{equation}
Y_A(S): A^{\ot d_1}\odo A^{\ot d_b}\to k
\end{equation}
which is given by the formula \eqref{asscoreq}. Note that the formula is invariant under cyclic rotations of the tensor factors at each boundary component, and is equivrariant with respect to renumbering the boundary components,see \S\ref{OTFTpar}.
The simplest OTFT correlation functions, which suffice to define the product, coproduct, pre-Lie and braces, are the $Y_A(P_{2n})$ where $P_{2n}$ is a $2n$-gon for which every other side is marked. The general correlation function  specializes to:
\begin{equation}
\label{polygoncoreq}
Y_A(P_{2n})(a_1\odo a_n)=
\la a_1\cdots a_n\ra
\end{equation}

\begin{rmk}
\label{OTFTrmk}
Usually OTFTs are defined as involutive functors $Z$ from a cobordism category, see e.g.\ \cite{KP,LaudaPfeiffer}. The Frobenius algebra is   $A=Z(I)$, $I=[0,1]$. The correlation function $Y(S)$ is the value of $Z$ on $S$ as a cobordism with all intervals being inputs and an empty output. As $Z(\emptyset)=k$ this gives the map above. Vice--versa, since the functor $Z$ is involutive, $A$ and the correlation functions fixes all of $Z$ up to equivalence. When specifying inputs and outputs, one has to be careful with the orders, this explains different versions of the Frobenius equation \eqref{frobeq}. Dualizing inputs and outputs yields the different forms discussed in \S\ref{calcpar}.
\end{rmk}

\subsubsection{Global correlators for a cell}
\label{cellorderpar}
Given $(c,\wt)$ represent each integer weighted arc $a\in \a$ with weight $p$ is represented by $p$ parallel arcs. These decompose the surface into sub--surfaces given by the complementary regions: $\Sigma =\bigcup_{v\in V} S_v$ where the intersections are at the boundaries of the $S_v$  along the arcs, see Figure \ref{coprodfig} for an example.
The set $V$ is the set of vertices of a dual description in terms of almost ribbon graphs, \cite[Appendix A1]{postnikov} and Figure \ref{figure8}.
Let $l_i$ be the number of intervals at boundary $\del_i\Sigma, i=1\kdk n$ marked by $1$ in $(c,\mk,\wt)$, then
\begin{equation}
\label{cwtcor}
Y_A(c,\mk,\wt):=\bigotimes_{v\in V}Y_A(S_v)\circ \sigma: A^{\ot l_1}\odo A^{\ot l_n}\to k
\end{equation}
where $\sigma$ is the shuffle that shuffles the factors of $A$ into their relative position.  We used indexing by sets to make the formula easier.
There is a natural order on $V$ given by enumerating each $S_v$ by the first appearance of an interval that belongs to it.
The intervals themselves, and hence the factors of $A$, are enumerated first by the boundary component and then in their natural orientation starting at the interval containing the marked point, see also \S\ref{standardpar}.
These homogenous components \eqref{cwtcor} sum up to a correlation function
\begin{equation}
\label{localsumeq}
Y_A(c,\mk)=\sum_{\wt} Y_A(c,\mk,\wt)
\end{equation}

\subsection{PROP cells and their action}
\label{propcellpar}
In order to obtain the relevant PROP
one partitions the boundaries of $\Sigma$ into inputs $In_1\kdk In_n$ and outputs  $Out_1\kdk Out_n$  enumerating them separately.
Furthermore, one restricts the arcs to run from input to output only and requires that every input boundary has at least one incident arc. This is the Sullivan   quasi PROP
$\Diioarc$.
Lastly, one retracts the cells to a normalized version by scaling the coordinates so that the sum of barycentic coordinates separately at each input boundary  is $1$.
Let $s_i$ be the number of arcs incident to $In_i$, then the retracted  cell $c_1$ is a product of simplices $\Delta^{s_1-1}\times \dots \times \Delta^{s_n-1}$.
These cells make up the cell complex called $\Diioarc_1$, see  \cite{hoch1} for details.

\subsubsection{Standard marking and action}
\label{standardpar}
Each PROP cell
has a standard marking dictated by the input/output designation, cf.\ \cite{hoch2}. All module intervals are active.
All inner input intervals are active and all inner output intervals are inactive.  This defines $Y_A(c_1)$.
The operation has degree $\dim(c_1)$ which is the number of  input intervals not containing the base point.
The main result for this complex is that the cellular chains have a dg--action on $CH^*(A,A)$ see\cite[Theorem B]{ hoch2}. Here a cell $c_1$ with $n$ inputs boundaries and $m$ output boundaries
 acts via the operation $op_{CH}(c_1):CH^{\ot n}\to CH^{\ot m}$ with the graded components given in \eqref{opeq}.

The standard order is as follows: The module variable is assigned to the module--interval. This is followed in the linear order of $\bar TA$ according to the following rules.
 The input intervals are enumerated in  opposite order to the orientation and the output intervals are enumerated according to the orientation.

\subsubsection{Standard decomposition}
\label{standarddecomppar}
There are several ways to find standard decompositions of the operations into standard operations. The most useful for $CH^*(A,A)$ being the following:

\begin{thm}\cite[Proposition 4.13]{hoch2}
All the operations of the Sullivan PROP or even those of moduli spaces are expressible in terms of shuffles, deconcatenation coproducts $\diamond$ (on $TA$) and integrals.
\end{thm}

This kind of decomposition can be rewritten easily in other contexts,
to lift operations to the various versions of $CH$.
If one has a different context, then one should simply keep track of the fact that in \eqref{standardeq} the leading tensor is the one from the coefficients and which form of the operation
defined by the integral in the Frobenius case is being used. See \S\ref{calcpar}.
Sample considerations are given in \S\ref{gendisccopar}, \S\ref{geocoprodpar}, \S\ref{bdgeopar}.
\begin{rmk}
The deconcatenation coproduct $\diamond$ is a coproduct for the following monidal structure on $A$-$Mod$-$A$: $M\boxtimes N= M\ot_k A \ot_k N$.
A similar coproduct appears in \cite{Weibel} as a cotensor product in a different, but maybe not un--related, context.
This is the coproduct for the inputs corresponding to the interval marking by $1$. On the outputs, the product is the simple tensor product. Using the unit $u:k\to A$,
 there is an embedding $M\ot N\to M\boxtimes N$, which is precisely the application of the degeneracy maps.

From a simplicial point of $\boxtimes$ corresponds to the join, see \S\ref{joinrmk} and \S\ref{fig8par}. The Joyal dual monoidal product ``$+$'' ---cf.\ e.g.\ \cite[Appendix B]{HopfPart1}
and \cite[\S3.5]{HopfPart1} for explicit formulas---
 is what is used on the outputs.
\end{rmk}

 \subsubsection{Remarks on PROP and quasi PROPs}

Although the relevant structures on the chain level are PROPs, that are strictly associative, on the topological level there are two complication.
The first is, that there is a rescaling involved. This is possible without penalty for the operad part as a global scaling and expressed as a bicrossed product with a scaling operad \cite{cact}.
For the multi-gluings in the PROP one has to perform local scalings and this results  in associativity only being up to homotopy, which is the content of the notion of quasi PROP \cite[Definition 5.22]{hoch1}.
The explicit homotopies are controlled by rather intricate flows on the geometric level  \cite{KP} that even work in the more general modular operad setting.
The second complications, which already appeared in the operad part, \cite{del,cyclic} is that in order to obtain a cell complex with cells of the right dimension,
one needs to retract to a smaller complex given by normalization, which is also a local scaling. Hence the second problem and the first one are of the same ilk.  Already the normalized operad is only a quasi operad \cite[1.1.1 Definition]{cact}.
The full statements for the topological level are contained in \cite[Theorem D]{hoch1}.

\subsubsection{Geometry of the PROPs on the topological level}
\label{arcgeopar}
The cell itself can be viewed in different ways as giving ``geometric actions''. Here it is helpful to regard the dual graph, see Figure \ref{figure8}. The blue graph is the image of the map $\Loop$ of \cite[Definition 4.3]{KLP}
which  identifies the points of the various in and output circles using a foliation, see \cite{KLP,KP} for more details.
This means that the outside circle gets identified with the figure 8 configuration in such a way that all the base points coincide and the length of the two parts is given by $1-t$ and $t$, yielding a 1--parameter family. The length in the picture is given via a partially measured foliation indicated by parallel lines.
The extra tails or spines give the base points and the dotted spine keeps track of  the polycyclic structure \cite[Appendix A]{postnikov} in which the extra tail pointing to the ``lone loop'' is a cycle by itself.
There is no extra genus, but an extra boundary component with one interval, which is a module-inverval.  Geometrically this means that there is a second constant loop that is identified with the input base point.
The polycyclic structure and extra markings appear in the combinatorial compactification of moduli space, \cite{hoch1} that  was axiomatized  with graphs in \cite{postnikov} using polycyclic graphs aka.\ stable ribbon graphs \cite{konteurope}. It also related to non--Sigma modular operads \cite{Marklnonsigma, decorated,BergerKaufmann,feynmanrep}. The extra decoration manifests itself in the action, which does  {not} only involve e polygon correlators.
The interpretation of the partially measured foliations as moving pieces of string according to \cite[\S5.11]{woods} is in Figure \ref{cellfig}, which also illustrates the time reversal.

This is also exactly the action that is induced on loop spaces algebratized in \S\ref{looppar}.  The totalization is the discretization of the map $\Loop$.
By \S\ref{geocoprodpar} and \S\ref{bdgeopar}
is exactly realized via an intersection interpretation of \S\ref{intpar}  co--simplicially on the loop spaces, cf.\ \S\ref{simppar}.

\section{Calculations}
\label{calculationspar}
\subsection{Correlators}
\label{calcpar}
We give the dualizations for the functions $\int_n:=\eps \circ \mu^{[n]}:A^{\ot n}\to k$ given by $\int_n a_1\odo a_n=\la a_1\cdots a_n\ra$ ---where $\mu_A^{[n]}$ is the iterated multiplication.
This defines different forms of the operation, and we discuss which  of these forms may exist,
without the Frobenius assumption.  These forms may break the cyclic symmetry.
The tensor factors may simply be a $k$-module $V$ or its dual $\check V$, an  associative algebra $A$, or  a coassociative coalgebra $\CA$, e.g.\ $\check A$ for $A$ finite dimensional or of finite type. We will tacitly assume such a condition, when we use $\check A$ as.a coalgebra variable. Subscripts  indicate modules, e.g. $\bisub{A}{M}{A}$ stands for an $A$-$A$-bimodule $M$.
 The idea is that there is a hierarchy of operations on tensors: shuffles, contractions, multiplication, comultiplication, actions. This is the point of
 view underlying \cite{Gerstenhaber} and \cite{hoch2}.

\subsubsection{Relaxations of the Frobenius condition for $\int_n$}
\label{viewpar}
The basic form $\int_{n+1}$ exists for an  associative algebra $A$ with a morphism $\eps:A\to k$.  This is cyclic if $\eps$ is a trace.
It  also exists  for an algebra $A$ as a morphism $ev\circ \mu_A^{[n]}:\check A\ot A^{\ot n}:\check a_0\ot a_1\odo a_n\mapsto \check a_0(a_1\cdots a_{n})$
 as the multiplication followed by the dual paring, aka.\ evaluation.  The latter form is basis of the action of the little discs \cite{del}, since the restriction on the surfaces says that all the regions are polygonal and have one distinguished coalgebra tensor.
Note,

Dualizing in the last slot generally yields the iterated multiplication map $\mu_A^{(n)}\in Hom(A^{\ot n},A)$: $\mu_A^{(n)}(a_1\odo a_n)=a_1\cdots a_n$ which is defined for any associative algebra $A$.
Dualizing all but the first entry yields the  iterated co--multiplication $\Delta_{\CA}^{(n)}\in Hom (\CA,\CA^{\ot n})$ which exists for any coalgebra $\CA$.
Dualizing all entries yields the element $\Delta^{(n+1)}(1)$ which exists for a pointed coalgebra $(\CA,1)$.
\subsubsection{Calculations for low $n$}
\label{explicitpar}
\begin{list}{(\arabic{lstcounter})}{\usecounter{lstcounter}\leftmargin=20pt}
\item
\label{int2}
$\int_2=\eta$: Interpreted as a morphisms $A\to A$ this is $id_A$. As a morphism $k\to A^{\otimes 2}$ this is $C=\Delta(1)$ that is the Casimir element dual to the form.
 \\
{\sc Restrictions:} The form $\eta$ is simply the dual paring and exists as the evaluation map $ev:\check V\otimes V\to k$. The form $id_V:V\to V$ only needs a $k$--module $V$.
 As $k\to \check V\ot \check V$ is is a bilinear form. In the form $k\to V\ot V$ it is simply a fixed element ---sometimes called a propagator---
which is needed to operadically compose correlation functions \cite[\S2]{hoch2}.

\item
\label{int3}
$\int_3$:
By dualizing in the third slot, this represents $\mu_A\in Hom(A^{\otimes 2},A)$.  \begin{equation}
a\ot b\mapsto
(\int_3 a\ot b\ot \Cp)\ot \Cpp
= \sum \la ab,\Cp\ra  \Cpp =
ab=\mu(a,b)
\end{equation}
By dualizing in the second and third slot this yields $\Delta_{\CA}\in Hom(\CA,\CA^{\otimes 2})$.
\begin{equation}
\begin{aligned}
a\mapsto&\sum_{C_1,C_2} (\int_3a \ot \Cp_1\otimes \Cp_2)\; \Cpp_1 \otimes \Cpp_2
=\sum_{C_1,C_2}  \la a, \Cp_1\Cp_2\ra \Cpp_1\otimes \Cp_2\\
& =\sum_{C_1,C_2} \la \Delta(a), \Cp_1\otimes \Cp_2\ra\;  \Cpp_1 \otimes \Cpp_2
 = \Delta(a)
\end{aligned}
\end{equation}

\item
\label{int4}
 $\int_4$: We will give the dualization in the 2nd and 4th slot. The map $A^{\ot 2}\to A^{\ot 2}$ is
\begin{equation}
\begin{aligned}
&a\ot b \mapsto \sum_{C_1,C_2}(\int_4 a \ot\Cp_1 \ot b \ot \Cp_2)\;
\Cpp_1\otimes \Cpp_2
=\sum_{C_1,C_2}\la a,\Cp_1 b \Cp_2\ra\, \Cpp_1 \otimes \Cpp_2\\
&=\sum_{C_1,C_2}\la\Delta_A(a),\Cp_1\otimes b \Cp_2\ra \;\Cpp \otimes \Cp_2
=\sum_{C_1,C_2}\la\Delta_A(a)(1\otimes b),\Cp_1\otimes \Cp_2\ra \;\Cpp_1 \otimes \Cp_2\\
&=\sum_{C_1,C_2}\la\Delta_A(a),(1\otimes b)(\Cp_1\otimes \Cp_2)\ra \;\Cpp_1 \otimes \Cp_2
=\Delta(a)(1\otimes b) =a\sw \ot a\sww b
\end{aligned}
\end{equation}
This form exists as a morphism $\CA_A \ot A\to \CA_A\ot {\CA}_A$.
i.e.\ $a$ is  a  decoation by a coalgebra element, where the coalgebra is a right $A$ module.
 and $b$ is am algebra element.
Using \eqref{deltaeq} also  $\Delta(a)(1\otimes b)=\D(1) (a \ot b)=C (a \ot b)$, where now there are no restrictions of $a,b$
at first, but there has to be some module structure. E.g. if $a,b\in A$ and $C\in \check A\ot \check A$
then this is a morphism $A\ot A\to \check A\ot \check A$ etc..  Switching the roles of $a$ and $b$, one also has the form: $A\ot \leftsub{A}{\CA}\to  \leftsub{A}{\CA}\ot  \leftsub{A}{\CA}$:
\begin{equation}
\label{homcoefeq}
a\ot b\mapsto ab\sww \ot b\sw=(a\ot 1)\D_{\CA}^{op}(b)
\end{equation}

\item
\label{int5}
$\int_5$, we compute the dualization in the 2nd and 4th slot: The map $A^{\ot 3}\to A^{\ot 2}$ is:
\begin{equation}
\begin{aligned}
&a\ot b\ot c\mapsto \sum_{C_1,C_2}(\int a \Cp_1 b \Cp_2 c) \,\Cpp_1 \otimes \Cpp_2=\sum_{C_1,C_2}\la \Cp_1 b \Cp_2 c, a\ra \, \Cpp_1 \otimes \Cpp_2\\
&=\sum_{C_1,C_2}\la(\Cp_1 b) \otimes (\Cp_2 c),\D_A(a)\ra \,\Cpp_1 \otimes \Cpp_2
=\sum_{C_1,C_2}\la \Cp_1 \otimes \Cp_2, (b\otimes c) \D_A(a)\ra \,\Cpp_1 \otimes \Cpp_2\\
&=(b\ot c)\Delta(a)=\sum ba^{(1)}\ot c a^{(2)}
\end{aligned}
\end{equation}
which exists as a morphism  $\leftsub{A}{\CA} \ot A\ot A\to \leftsub{A}{\CA}\ot \leftsub{A}{\CA}$.

From \eqref{deltaeq} it follows that:
$
 (b\ot c)\Delta_A(a)=(b\ot c)\D_A(1)(a\ot 1)=(b\ot ca)\D_A(1)
$.
\end{list}
\subsection{OTFT from a Frobenius algebra}
\label{OTFTpar}
We will now show that Assumption 4.1.2 \cite{hoch2} of commutativity of $A$ is not necessary and that the equations of Remark 4.2 \cite{hoch2} hold for any Frobenius algebra. Let $A$ be a Frobenius algebra
as in \S\ref{frobpar}.
Since $\int$ is cyclic, we have that
\begin{equation}
\label{cyceq}
\forall 1\leq i,j\leq n: \la a_i\cdots a_n a_1\cdots a_{i-1}\ra=\la a_j\cdots a_n a_1\cdots a_{j-1}\ra
\end{equation}
Using that $a=\sum  \la a\Cp \ra \Cpp$, we get the factorization
\begin{equation}
\label{spliteq}
\la a_1\cdots a_n\ra =\sum \la a_1\cdots a_i\Cp\ra\la\Cpp a_{i+1}\cdots a_n\ra
\end{equation}

\begin{prop}
\label{cyclicprop}
 Let $A$ be a Frobenius algebra. Using the notation of \S\ref{frobpar} and the one above:
For all $1\leq i,j\leq n$ and $1\leq k,l\leq m$:
\begin{equation}
\label{annuluseq}
\begin{aligned}
&\sum \la a_1\cdots a_i\Cp_1 b_k\cdots b_mb_1 \cdots b_{k-1}\Cpp_1 a_{i+1}\cdots a_n\ra\\
&=\sum\la a_1\cdots a_j\Cpp_2 b_l\dots b_m b_1\cdots b_{l-1} \Cp_2 a_{j+1}\cdots a_n\ra
\end{aligned}
\end{equation}
Also,
\begin{equation}
\begin{aligned}
\label{toruseq}
 & \sum \la a_1\cdots a_i\Cp_1a_{k+1}\cdots a_l\Cp_2 a_{j+1} \cdots a_k\Cpp_1 a_{i+1}\cdots a_j\Cpp_2 a_{l+1}\cdots a_n\ra\\
&=  \sum \la a_1\cdots a_n\Cp_1\Cp_2 \Cpp_1  \Cpp_2\ra
\end{aligned}
\end{equation}
\end{prop}

This fact is well known, albeit maybe not in this presentation,
as it is equivalent the the theorem that 2d Open Topolgical Field Theories are
equivalent to Frobenius algebras, see Remark \ref{OTFTrmk}
The two equations correspond to cuts for the annulus and the torus with one boundary, see Figure \ref{gluefig}.
\begin{figure}
\includegraphics[width=.65\textwidth]{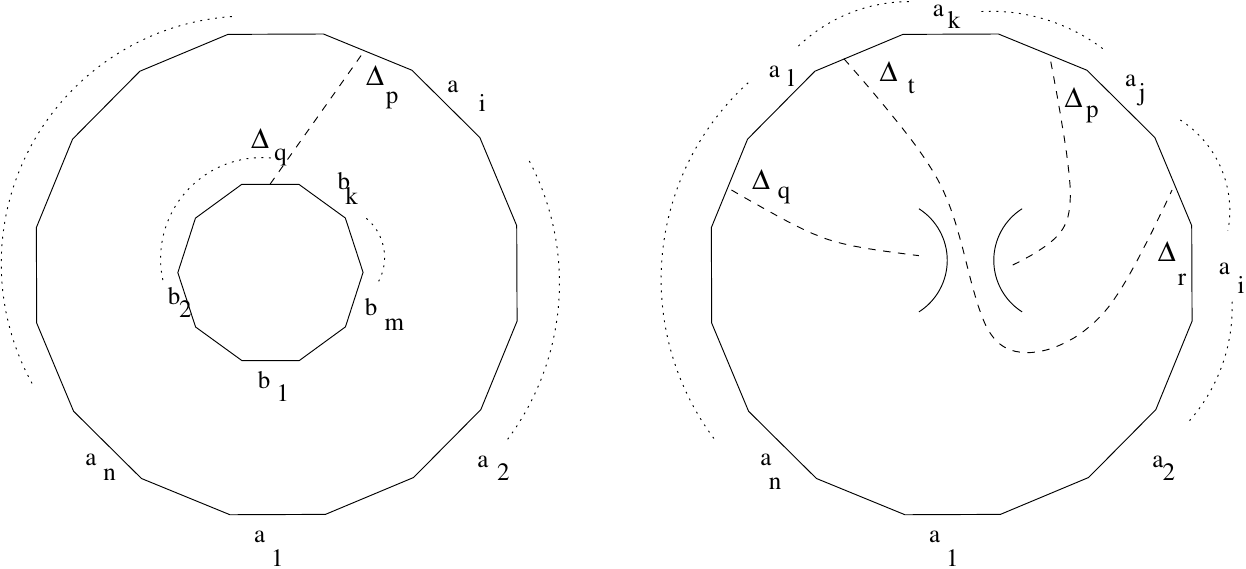}
\caption{\label{gluefig} The cut on the annulus corresponding to \eqref{annuluseq} and the cuts on the torus with one boundary component \eqref{toruseq}.
The equations say that the choice of endpoints of
the cuts does not matter.
}
\end{figure}

\begin{proof}  For \eqref{annuluseq} Assume wlog $i<j$ and $k<l$
\begin{equation*}
\begin{aligned}
&\sum\la a_1\cdots a_i\Cp_1 b_k\cdots b_mb_1 \cdots b_{k-1}\Cpp_1 a_{i+1}\cdots a_n\ra\\
&=\sum_{C_1,C_2}\la a_{j+1}\cdots a_na_1\cdots a_i\Cp_1 b_k\cdots b_{l-1}\Cp_2\ra
\la\Cpp_2
b_l \cdots b_{m}b_1\cdots b_{k-1}\Cpp_1 a_{i+1}\cdots a_j\ra\\
&=\sum\la a_1\cdots a_j\Cpp_2 b_l\dots b_m b_1\cdots b_{l-1} \Cp_2 a_{j+1}\cdots a_n\ra
\end{aligned}
\end{equation*}
where in the first step we used \eqref{cyceq} and then \eqref{spliteq} to first rotate until $a_j$ is at the end and then split.
In the second step, we rotated both expressions with \eqref{cyceq} so that $\Cp_1$ is on the right and $\Cpp_1$ is on the left
and then used \eqref{spliteq} to merge them.
For \eqref{toruseq}
\begin{eqnarray*}
&&  \sum_{C_1,C_2}\la a_1\cdots a_i\Cp_1a_{k+1}\cdots a_l\Cp_2 a_{j+1} \cdots a_k\Cpp_1 a_{i+1}\cdots a_j\Cpp_2 a_{l+1}\cdots a_n\ra\\
&=&  \sum_{C_1,C_2}\la a_1\cdots a_i a_{i+1}\cdots a_j\Delta_ p a_{l+1}\cdots a_n\Cp_2 a_{k+1}\cdots a_l\Cpp_1a_{j+1} \cdots a_k \Cpp_2\ra\\
&=& \sum_{C_1,C_2} \la a_1\cdots a_i a_{i+1}\cdots a_ja_{j+1} \cdots a_k \Cp_1\Cp_2 a_{l+1}\cdots a_n\Cpp_1a_{k+1}\cdots a_l \Cpp_2\ra\\
&=& \sum_{C_1,C_2} \la a_1\cdots a_i a_{i+1}\cdots a_ja_{j+1} \cdots a_ka_{k+1}\cdots a_l\Cp_1 \Cp_2\Cpp_1  a_{l+1}\cdots a_n\Cpp_2\ra\\
&=& \sum_{C_1,C_2}\la a_1\cdots a_i a_{i+1}\cdots a_ja_{j+1} \cdots a_ka_{k+1}\cdots a_la_{l+1}\cdots a_n\Cp_1\Cp_2\Cpp_1  \Cpp_2\ra\\
\end{eqnarray*}
where we used used \eqref{annuluseq} to move each block not yet in place by one in each step.
\end{proof}
\begin{cor}
\label{removeassumptioncor}
The assumption of commutativity \cite[Assumption 4.1.2]{hoch2} is unnecessary and  all the operations of \cite{hoch2} are defined for the Hochschild cochains of any associative Frobenius algebra $A$. The local correlation function $Y(S)$ in (4.3) of \cite{hoch2} for a surface $S$ of genus $g$ with $b$ decorated boundaries  where the $i$-th boundary is decorated by elements $a_1^i,\dots, a^i_{n_i}$  the non--commutative case is:
\begin{equation}
\label{asscoreq}
Y(S)(\bigotimes_{j=1}^b\bigotimes_{i=1}^{n_j}  a_i^j)=
\la (a_1^1\cdots a_{n_1}^1)\prod_{l=2}^{b} \left(\sum\Cp_l a_1^l\cdots a_{n_l}^l \Cpp_l
  \right)\prod_{k=1}^g\left(  \sum_{C_{k_1}C_{k_2}} \Cp_{k1}\Cp_{k2}\Cpp_{k1}\Cpp_{k2}
 \right)\ra
\end{equation}
which is simply the correlation function $Y_A(S)$ of the 2d--OTFT of the marked surface for the OTFT defined by $A$.
\end{cor}
\begin{proof}
The operations {\it a priori} depend on a choice of triangulation by extra arcs/cuts. Since the two equations \eqref{annuluseq} and \eqref{toruseq} hold, the result is independent of such a choice as they can be used to put the cuts into a standard position yielding \eqref{asscoreq}. This follows from the transitivity of Whitehead moves on triangulations. By the gluing axioms of an OTFT this is the correlation function corresponding to the given surface.

\end{proof}
Note that \eqref{asscoreq} seems to depend on the enumeration of the boundary components but the result is independent of that ordering, again by applying \eqref{annuluseq}. By the same equation, it also only depends on the cyclic order of the elements at each boundary. There is a standard order of all the elements given by the fact that the boundary components are labelled.

\subsubsection{Pseudo-commutative Frobenius algebras}
We call $A$ pseudo--commutative if $\mu\Delta(ab)=e ab$.
\begin{lem}
If one of the following conditions holds,  $A$ is pseudo--commutative:
(1) $A$ is commutative, or (2) $A$ is  graded Gorenstein, or  (3)  $\Delta_A(1)(a\ot 1)=\Delta_A(1)(1\ot a)$, or (4)  $(1\ot a)\D_A(1)=(a\ot 1)\D_A(1)$.
\end{lem}
\begin{proof}
$A$ is pseudo--commutative iff  $\sum \Cp a \Cpp b=\sum \Cp\Cpp ab$, which is the case if $A$ is commutative.
If $A$ is graded Gorenstein of degree $d$ then degree $\sum \Cp a \Cpp b$  is of degree at least $d$ and unless  $deg(a)=deg(b)=0$  both sides are $0$.
As all elements in degree $0$ lie in the center, the equation holds.

For (3) using \eqref{deltaeq}:
 $\sum \Cp a \ot \Cpp b= \D_A(1)(a\ot 1)(1\ot b)=\D_A(1)(1\ot a)(1\ot b)=\D_A(1)(1\ot ab)=\sum \Cp\ot \Cpp ab$ which after applying $\mu$ to both sides yields the defining equation.
If $A$ is graded Gorenstein of degree $d$ then degree $\sum \Cp a \Cpp b$  if of degree at least $d$ and unless  $deg(a)=deg(b)=0$ or both sides are $0$.
As all elements in degree $0$ lie in the center, the equation holds. The case (4) is similar.

\end{proof}
\begin{lem}
In case that $A$ is pseudo--commutative, equation (4.3) of \cite{hoch2} holds, that is
\begin{equation}
\label{comcoreq}
Y(S)(\bigotimes_{j=1}^b\bigotimes_{i=1}^{n_j}  a_i^j)=
\la\prod_{j=1}^b\prod_{i=1}^{n_j}  a_i^j e^{-\chi(S)+1}\ra
\end{equation}
\end{lem}
\begin{proof} If $A$ is pseudo--commutative, we can move all the factors $\Cp_i\Cpp_i$ next to each other to the right.   $\Cp\Cpp=\mu\Delta(1)=e$ and there are $b-1+2g=-\chi(S)+1$ such factors.
\end{proof}
\begin{rmk}
\label{vanishrmk}
 Note that if $A$ is graded Gorenstein, for degree reasons, \eqref{comcoreq} is $0$ unless $-\chi(S)+1\leq 1$ and if $\chi(S)=0$, that is $S$ is an annulus, then all the $a_i^j$ must be of degree $0$, so that in this case, the correlation the function vanishes modulo the constants $A_0$.
 \end{rmk}

\subsubsection{Stabilization and the semi--simple case}

We call $A$ {\em E-unital} if $e=1$. In this case
 $\eps(a)=tr(L_a)$ where $L_a$ is the left multiplication by $a$, as $\eps(a)
 =\sum \la a, \Cp\Cpp\ra=\sum \la a\Cp,\Cpp\ra=\eps(tr(L_a))$.

 \begin{lem}
$A$ is commutative and $E$--unital if and only if $A$ is {\em isometric}, i.e.\ $\Delta_A\mu_A= id_A$.
 \end{lem}
 \begin{proof}
``$\Rightarrow$'': \eqref{deltaeq} implies that if $A$ is commutative and E-unital, it is isometric.\\
``$\Leftarrow$'': By the assumption $1=\mu\Delta(1)=e$ and using \eqref{deltaeq}, $\mu(\Delta_A(a))=\mu(\Delta_A(1)(a\ot 1))=\sum \Cp a \Cpp=a$. Using this and  Proposition \ref{cyclicprop}, $\la ab,c\ra=\sum \la \Cp ab\Cpp c\ra=\sum \la \Cp b a\Cpp c\ra=\la ba,c\ra$.
 \end{proof}
\begin{rmk}
 If $A$ is isometric, then the operations pass through the stabilization \cite{postnikov} and contain an $E_{\infty}$ structure \cite{ROMP}. \end{rmk}

 \begin{rmk}[Semi-simplicity]
If $A$ is free of finite rank and semi--simple, there is a basis $e_i$ with $e_ie_j=\delta_{ij}e_i$. This implies that $\sum_i e_i=1$. Setting $\lambda_i=\eps(e_i)$, $e^i=\frac{1}{\lambda_i}e_i$, and  $e=\sum_i \frac{1}{\lambda_i} e_i$ is invertible with inverse
$e^{-1}=\sum_i \lambda_i e_i$.
If all $\lambda_i=1$, which is sometimes called normalized semi--simple, then  $A$ is E--unital.
In case $A$ is semi--simple, there is a flow to a normalized $A$, see \cite{ROMP}.
In \cite{Abr} it is shown that if $A$ is even commutative, then being semi-simple is equivalent to $e$ being invertible.
\end{rmk}

\section{Duallities and further topics}
\label{furthertopicspar}
\subsection{Dualities}
\label{dualitypar}

\subsubsection{Na\"ive duality}
\label{naivedualpar}
The operations were defined by dualizing the  arguments of $\hat Y$, see \S\ref{conversionpar}.
The choice of inputs and outputs is dictated by the cell. One can ask about the other forms of the operation $\hat Y_{CH}$ as $\bar TA$ is graded isomorphic to its dual.
As operations these always exists, but their PROP structure is more complicated, see \S\ref{KWangpar}.

Switching all inputs to outputs  for the PROP one obtains an na\"ive input/output dual operation that is an $(m,n)$-ary operation from every $(n,m)$-ary operation via $Hom(CH^{\ot n},CH^{\ot m})\simeq
Hom(CH^{\ot m+n},k)\simeq Hom(CH^{\ot m},CH^{\ot n})$.

\begin{ex}
For instance, the degree $0$ product is dual to a degree $0$ coproduct, which is different from the natural degree $1$ product. It corresponds to the pointwise co-product, of \cite{Sopenclosed}.
Similarly the degree $1$ coproduct is dual to a degree $1$ product which shares the same correlation function \eqref{Ycopeq}. This is the sum over the products $\sqcup$, cf.\ \cite[\S4.1.1, eq.\ (4.10)]{hoch2},  where the degree $n\mapsto (p,q),n=p+q+1$ part of the coproduct dualized to $\sqcup$ in degree $(p,q)\mapsto p+q+1=n$.
\begin{equation}
f\sqcup g(a_1\odo a_n)=f(a_1\odo a_p)a_{p+1}g(a_{p+2}\odo a_n)
\end{equation}
 To obtain the usual cup product, one needs to apply a degeneracy, that is set $a_{p+1}=1$.
\end{ex}

\subsubsection{Time reversal symmery (TRS)}
\label{TRSdualpar}
Given a cell  $c$ represented by an arc family and a boundary input output marking, we define the TRS dual $\check c$ by reversing the ``in'' and ``out'' labels.
 This  {\em changes the normalized cell} and the interval marking in the discretized PROP.

For a cell $c$ with arcs only from inputs to outputs in which all boundaries are hit, that is a cell of $\ioarc$ in the notation of \cite{hoch1}, the TRS dual $\check c$ is also a cell of $\ioarc$, and hence retracts to a normalized cell $\check c_1$ of  $\Diioarc_1$. If $c_1$ is the normalized cell of $c$ then the TRS dual of the operation $op(c_1)$ is $TRS(op(c_1)):=op(\check c_1)$.
Unlike the na\"ive dual, the  TRS dual operation usually has different degree.
The degree of a normalized cell with $n$ inputs and $m$ outputs, and thus the degree of the corresponding operation, is $\#$ arcs $-$ $n$.
 The degree of the TRS dual  which is an $m$ to $n$ operation is  $\#$ arcs $-$ $m$. Thus the degree difference of the operations is $m-n$.

\begin{rmk}
\label{joinrmk}
This can simplicially be understood as two different join decompositions. A cell $c=(\Sigma, \alpha)$ of $\Diioarc$     is a $\Delta^{|\a|-1}$. If $\Sigma$ has $n$ inputs and $m$ outputs, then there are two partitions of $\a$ $\a=\amalg_{i=1}^n \a_i$ and $\a=\amalg_{j=1}^m\a'_j$. This gives rise to two join decompositions $[\a_1-1]*\dots *[\a_n-1]
= [\a-1]=[\a'_1-1]*\dots *[\a'_m-1]$. The normalization drops the $*$ and replaces it with the polysimplicial product.
\end{rmk}

\begin{ex}
The cell $c_\D$ for the comultiplication is the TRS dual of the cell $c_\mu$ for the multiplication, which has $2$ arcs, see Figure \ref{cellfig}.
The $(2,1)$ multiplication of degree $2-2=0$ has as TRS dual the $(1,2)$ comultiplication of degree $2-1=1$.
\end{ex}

Several interesting cells appear as homotopies for the Gerstenhaber and BV structure \cite{KLP}. Their TRS duals give new interesting homotopies for the TRS dual operations.
\begin{ex}
\label{dopex}
The TRS dual for the pre-Lie or Gerstenhaber product gives a homotopy relation for $\D$ and $\D^{op}$.
$
\Delta_{CH}+\Delta^{op}_{CH}\sim C
$
where $C$ only has components $C:CH^n\to CH^{n-1}\ot CH^0$ which,  using notation as in Theorem \ref{coprodthm}, are given by
\begin{equation}
\begin{aligned}
& C(f^n)(a_1\odo a_{n-1}\ot \lambda)
=\sum_{i=1}^n \pm \lambda f(a_1\odo, a_{p-1},\Cp_1,\a_{p}\odo a_{n-1})\ot \Cpp_2\Cpp_1\Cp_2
\end{aligned}
\end{equation}
This follows from
applying the general procedure laid out in \S\ref{cellpar} to Figure \ref{Deltahomotopyfig}.
The cell given in Figure \ref{Deltahomotopyfig} gives that homotopy of the sum of the two operations to the operation of the base side.
 Since the boundary of the operation $C$ and $\Delta+\Delta^{op}$ is $0$ they are  cohomology operations.

 Note that in the graded Gorenstein case, by degree reasons, operation equals
\begin{equation}
\sum_{i=1}^n \pm \lambda f(a_1,\odo, a_{p-1},e_{top},\a_{p}\odo a_{n-1})\ot e_{top}\end{equation}
This is a sort of Poincar\'e dual degeneracy map,
which geometrically corresponds to spawning off a loop at some point of the loop.
This operation is zero in the reduced $\widetilde{CH}^*(A,A)$ in the Gorenstein case.
 \begin{figure}
  \includegraphics[align=b,width=.6\textwidth]{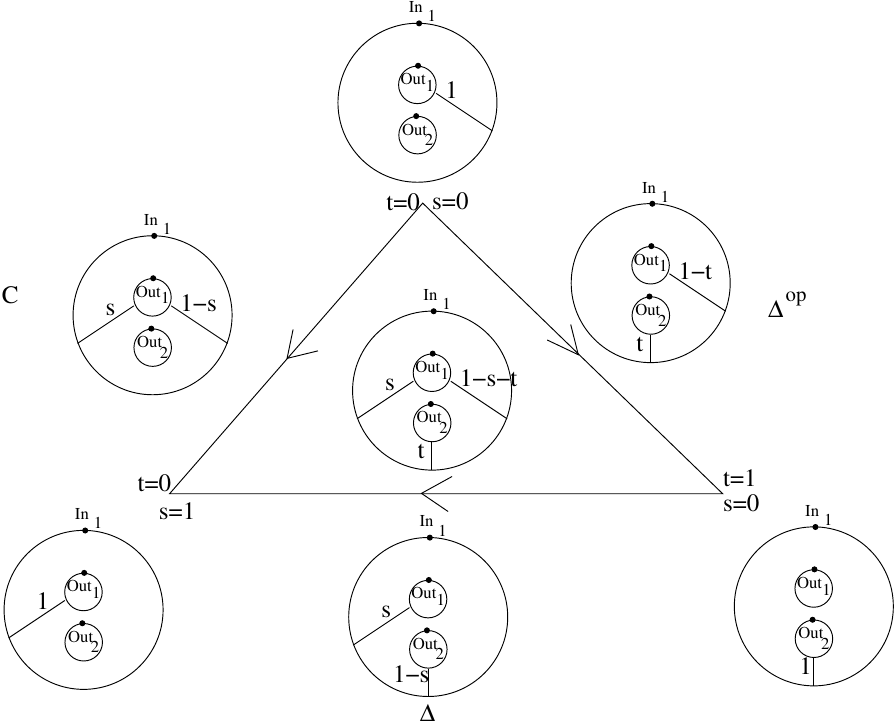}\hskip .7cm
  \includegraphics[align=b,width=.14\textwidth]{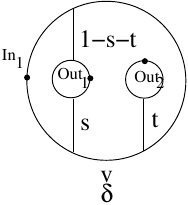}  \hskip .7cm
  \includegraphics[align=b,width=.12\textwidth]{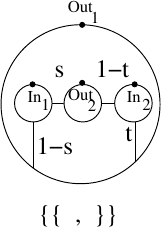}
\caption{\label{Deltahomotopyfig} The TRS dual of the pre--Lie product as a homotopy, the TRS dual of the operator $\delta$ and the double bracaket.}
\end{figure}
\end{ex}

\begin{ex}  Using the TRS dual of $\delta$ of \cite[\S2.2, Figure 7]{KLP}, see Figure \ref{Deltahomotopyfig} one obtains the following relation for the BV operator, see Figure \ref{animatedfig} ---denoted by BV here to avoid confusion with the coproduct:
$\Delta_{CH}(BV(f))= \check\del(f)+\tau_{12}\check\del(f)$, where
\begin{equation}
\begin{aligned}
\hat Y(\check \delta)(a_0\odo a_n)
 =&\sum_{p=1}^{n-2}\sum_{m=0}^{n-2}\pm
\eps(a_0) ( a_p\Cp \ot a_1 \odo a_{p-1} \ot a_{p+m+2}\odo a_n)\\
&\ot (a_{p+m+1}\Cpp \ot a_{p+1}\odo a_{p+m})
 \end{aligned}
\end{equation}
\end{ex}

Note that the TRS dual of the homotopies for the Gerstenhaber structure  and the BV property \cite[Figures 10,11,12]{KLP} should also yield interesting operations.
\subsubsection{Treating empty boundaries}

More generally, there can be empty output boundaries allowing to
``bubble off constant loops'', while ``in'' boundaries all  have to be hit.
Upon reversal, this condition gets switched, to all ``out'' boundaries are hit.
The correlation function is well defined as well in this case, by using the standard marking.
This may lead to  additional factors of $A$, which in the loop space operations stem from the inclusion of constant loops.
 These operations are also in the TRS dual  PROP where the incidence conditions on the arcs on the input and outputs are switched.
 The TRS dual operations  and correlations functions are summarized in Table \ref{TRStable}

\begin{table}
\begin{tabular}{l|l|l}
&in$\to$ out&out $\to$ in\\
\hline
No empty boundary&$Hom(CH^{\otimes n},CH^{\otimes m})$&$Hom(CH^{\otimes m},CH^{\otimes n})$\\
Correlation function&$Hom(CH^{\otimes n+m},k)$&$Hom(CH^{\otimes n+m},k)$\\
$r$ empty boundaries&$Hom(CH^{\otimes n},CH^{\otimes m}\otimes A^{\otimes r})$&$Hom(CH^{\otimes m}\otimes A^{\otimes r},CH^{\otimes n})$\\
&$\subset
(CH^{\otimes n},CH^{\otimes m+r})$&$\subset Hom(CH^{\otimes n+r},CH^{\otimes m})$\\
Correlation function&$Hom(CH^{\otimes n+m}\otimes A^{\otimes r},k)$&$Hom(CH^{\otimes n+m}\otimes A^{\otimes r},k)$\\
&$\subset Hom(CH^{\otimes m+n+r},k)$&$\subset Hom(CH^{\otimes m+n+r,}k)$
\end{tabular}
\caption{\label{TRStable} The TRS duals of operations and their correlation functions.}
\end{table}


\subsection{Further topics}
\label{furtherpar}
Note that in this setting the intervals between parallel arcs all belong to quadrilaterals and the relevant form of correlation function $\int_2$ is $id$ passing on the variable, see \S\ref{explicitpar}\eqref{int2}.
The interesting part of the action is therefore on the surfaces that are defined by the original, not replicated, arc system.  The original choice  is to use OTFTs and the pairing, we will briefly discuss other choices.
A fuller discussion is relayed to \cite{branetop}

\subsubsection{Animation}

In \cite{tsyganbook} a so called animation is introduced. This is naturally incorporated into the present framework.
 Given and $A$ module $ {M}$ and a  morphism $f:A\to A$, one has the new module structure $\rho_f(a,m)=f(a)m$.		
In this way one can twist the coefficient bimodule  $\bisub{A}{M}{A}$ by $f$. Furthermore one can act on $TA$ by powers of $f$ thus inducing new twisted Hochschild complexes.
In the presented formalism, this simply means allowing to replace the propagator for the quadrilaterals  to be given by $f$. In particular, given a set if maps $\mathcal{F}$,  we define an $\CF$--marking for
an arc system to be a map $f:\a\to \CF$. We define the operation for a $\CF$ marking for a cell $c$ to be given by $f(a)$ on the quadrilaterals.
In the general case, one should use self--adjoint maps $f$ and set the marked quadrilateral function to be $\la a, f(b)\ra$ on quadrilaterals marked by $f$. The form which used to be $id_A$ then becomes $a\mapsto f(a)$.
Figure \ref{animatedfig} shows the twisted $B$ operator which is a homotopy from $id_{A \ot TA}$ to $id_A\ot f^\ot$ and an example.
The operation is
\begin{equation}
a_0\odo a_n\mapsto \pm \sum \eps(a_0) \, a_p \ot f(a_{p+1})\odo f(a_n)\ot a_1\odo a_{p-1}
\end{equation}

\begin{figure}
 \includegraphics[width=.9\textwidth]{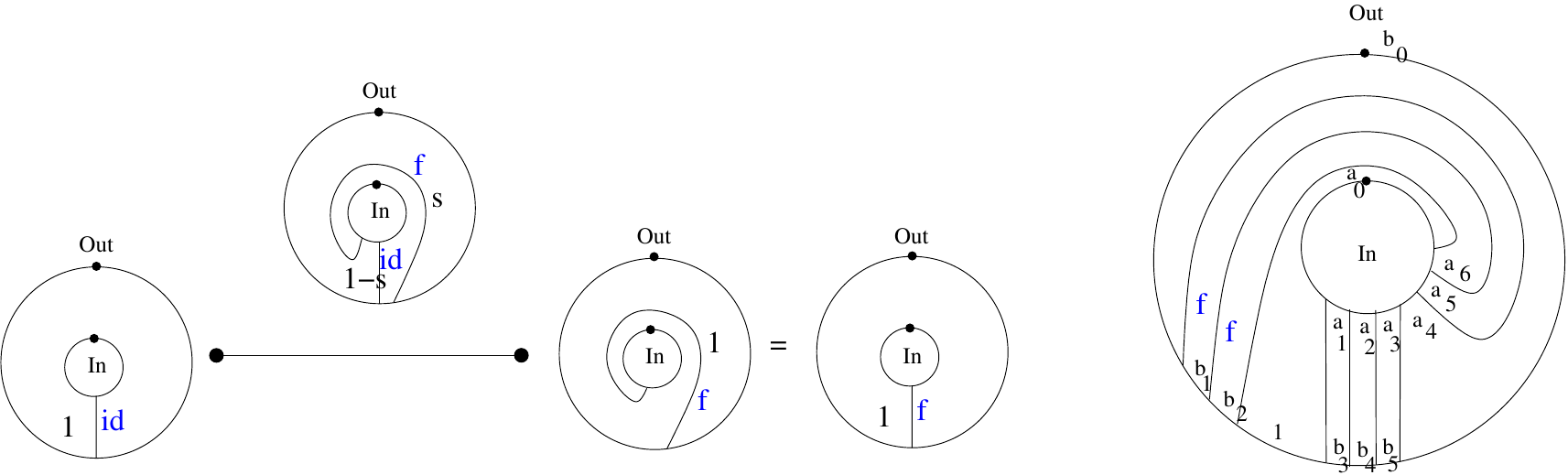}
\caption{\label{animatedfig} The animated BV as a homotopy and one discretized operation. The example is the morphism $a_0\odo a_6\mapsto \eps(a_0) \, a_4\ot f(a_5)\ot f(a_6)\ot a_1\ot a_2\ot a_3$.  Omitting reference to $f$ is the original BV}
\end{figure}

\subsubsection{Dualization to Hochschild chain operations}
\label{KWangpar}
Following discussions with Z.\ Wang, we can also look at the dualization to $CH_*(A,A)$. This allows to
reinterpreting the na\"ive  duality as an additional coloring for the PROP, and includes the products on Hochschild chains as found in \cite{Abb,RivWang} into our package. For this one labels na\"ively dualized inputs and outputs by $ho$ for homological and the ones keeping their original input/output designation as $co$ for cohomological. This yields a bi--colored dg--PROP which acts on Hochschild chains, via the $ho$ color, and Hochschild cochains via the $co$ color. The correlation functions are the {\em old} correlation functions of \cite{hoch2}. The decoration is according the both in/out and ho/co marking. $Co$ outputs and $ho$ inputs are decorated in the induced orientation while $co$ outputs and $ho$ inputs are decorated in the opposite of the induced orientation.

In \cite{KWang}, we furthermore show that the action on the Tate--Hochschild complex \cite{RivWang} can be subsumed into the formalism of  correlation functions \cite{hoch2} by a coloring keeping track of dualizations.
This allows us to realize the homotopy transfer concretely.
We naturally obtains  the operations that are found in \cite{Abb,RivWang}. For instance dualizing the coproduct from co to ho colors for all three boundaries, one recovers the degree $1$, $(2,1)$ product on $CH_*$ given by the formula
\begin{equation}
b_0\odo b_p\cup c_0\odo c_{n-p-1}=\pm \sum b_0\Cp\ot c_1\odo c_{n-p-1}\ot  c_0\Cpp \ot b_1\odo b_p
\end{equation}
which can be read off \eqref{Ycopeq}.

Another upshot is a nice interpretation of the mixed $m_3$ products $CH^*\ot CH_*\ot CH^*\to CH^*, CH_*\ot CH^*\ot CH_*\to CH_*$ in terms of the dualization of a double bracket, which arises as a natural homotopy incorporating the coproduct and its opposite simultaneously and is a Gerstenhaber double bracket
operation  $CH^{\ot 2}\to CH^{\ot 2}$ in the sense of  \cite{turaevbracket,vdBdouble}.
 The operation is given in Figure \ref{Deltahomotopyfig}. The action is given by
 \begin{equation}
 \begin{aligned}
 (a_0\odo a_n)&\ot (b_0\odo b_m)\mapsto
 \sum_{p,q} \pm \la a_p,b_q\ra (\Cpp b_0 \ot a_1\odo a_{p-1}
 \ot b_{q+1}\odo b_m)\\& \ot (\Cp a_0\ot  b_1\odo b_{q-1}\ot a_{p+1}\odo a_n)
 \end{aligned}
\end{equation}

\subsubsection{$A_\infty$--version} In \cite{KSchw} we showed that one can relax the condition of $A$ being associative to $A_\infty$
for the brace operations, aka.\ $A_\infty$--Deligne conjecture, and in \cite{Ward} the same was done tor the BV operators, aka.\ cyclic $A_\infty$ conjecture. As described in \cite{GSA}, this corresponds to introducing diagonals into the non--quadrilateral surfaces to specify an $A_\infty$ version of the OTFT. This type of different theory for the $S_v$ defined by $\a$ can be treated quite generally \cite{ddec2}.
There should be a nontrivial relation to the $A_{\infty}$ case to the   double brackets  above and those of \cite{KontVlass}.

\bibliography{hochnotebib}
\bibliographystyle{../TexInput/halpha}

\end{document}